\documentclass{birkjour_t2}
\usepackage{amsmath}
\usepackage{tkz-euclide}
\usetkzobj{all}
\tikzset{elegant/.style={smooth,thick,samples=50,cyan}}
\tikzset{eaxis/.style={->,>=stealth}}
\usepackage{color,tikz}
\usetikzlibrary{decorations.pathreplacing,calc}
\usepackage{rotating}
\tikzset{liltext/.style={font=\tiny}}
\usepackage{amsfonts}
\usepackage{amssymb}
\usepackage[english]{babel}
\usepackage{color}
\usepackage{graphicx}
\usepackage{lmodern}
\usepackage{amsthm}
\usepackage{mathrsfs}
\usepackage{microtype}
\usepackage{mathscinet}
\usepackage{enumitem}
\usepackage[cal=boondoxo,bb=ams]{mathalfa}
\usepackage{hyperref}
\hypersetup{hidelinks}

\renewcommand{\theequation}{\arabic{section}.\arabic{equation}}

\newtheorem{defn}{Definition}[section]
\newtheorem{thm}{Theorem}[section]
\newtheorem{prop}{Proposition}[section]
\newtheorem{lem}{Lemma}[section]

\newtheorem{rem}{Remark}[section]

\newcommand{\ml}{\mathcal}

\begin{document}

%
%
%
%
%
%
%
%

\title[Semilinear strongly damped wave equations in an exterior domain]
 {Blow-up of solutions to semilinear strongly damped wave equations with different nonlinear terms in an exterior domain}

\author[W. Chen]{Wenhui Chen}
\address{Institute of Applied Analysis, Faculty of Mathematics and Computer Science\\
	 Technical University Bergakademie Freiberg\\
	  Pr\"{u}ferstra{\ss}e 9\\
	   09596 Freiberg\\
	    Germany}
\email{wenhui.chen.math@gmail.com}

\author[A.Z. Fino]{Ahmad Z. Fino}

\address{Department of Mathematics, Faculty of Sciences\\
	 Lebanese University\\
	P.O. Box 826\\
	Tripoli\\
	Lebanon}
\email{ahmad.fino01@gmail.com; afino@ul.edu.lb}


\date{October 13, 2019}

\begin{abstract} 
	In this paper, we consider the initial boundary value problem in an exterior domain for semilinear strongly damped wave equations with power  nonlinear term of the derivative-type $|u_t|^q$ or the mixed-type $|u|^p+|u_t|^q$, where $p,q>1$. On one hand, employing the Banach fixed-point theorem we prove local (in time) existence of mild solutions. On the other hand, under some conditions for initial data and the exponents of power nonlinear terms, the blow-up results are derived by applying the test function method.
\end{abstract}

\maketitle

\section{Introduction}
\setcounter{equation}{0} 
In the present paper, we study the initial boundary value problem of the semilinear strongly damped wave equations in an exterior domain, namely,
\begin{equation}\label{eq1}
\begin{cases}
u_{tt}-\Delta u -\Delta u_t =f(u,u_t), &x\in \Omega,\,t>0,\\
u(0,x)=  u_0(x),\,u_t(0,x)=  u_1(x),&x\in \Omega,\\
u=0,&x\in \partial\Omega,\,t>0,
\end{cases}
\end{equation} 
where $\Omega\subset\mathbb{R}^n$ is an exterior domain whose obstacle 
$\mathcal{O}\subset\mathbb{R}^n$ with $n\geqslant 1$ is bounded with smooth compact boundary $\partial\Omega$. Without loss of generality, we assume that $0\in\mathcal{O}\subset\subset B(R)$, where $B(R):=\left\{x\in\mathbb{R}^n:\;|x|< R\right\}$ denotes a ball with radius $R$ centered at the origin. Precisely, the nonlinear terms on the right-hand sides of the equation in \eqref{eq1} can be taken by the forms
\begin{align}\label{Nonlinearity}
f(u,u_t)=|u_t|^q\quad\mbox{or}\quad f(u,u_t)=|u|^p+|u_t|^q
\end{align}
with $p,q>1$. Our main goals in this paper are to derive local (in time) existence of mild solution and blow-up of solutions under some assumptions on initial data and the exponents $p,q$. Especially, we are interested in the combined effects and the interplay between the power nonlinearity $|u|^p$ and nonlinearity of derivative-type $|u_t|^q$. 

Let us recall some results related to our problem \eqref{eq1}. In recent years, the strongly damped wave equation has caught a lot of attention. 

First of all, we consider the Cauchy problem for strongly damped wave equations. Concerning the linearized Cauchy problem 
\begin{equation}\label{Eq.Linear.SDW}
\begin{cases}
u_{tt}-\Delta u-\Delta u_t=0,&x\in\mathbb{R}^n,\,t>0,\\
u(0,x)=  u_0(x),\,u_t(0,x)=  u_1(x),&x\in \mathbb{R}^n,
\end{cases}
\end{equation}
the papers \cite{Ponce1985} and \cite{Shibata2000} derived $L^p-L^q$ estimates not necessary on the conjugate line. Later, asymptotic profiles of solutions to \eqref{Eq.Linear.SDW} in a framework of weighted $L^1$ data were obtained in  \cite{Ikehata2014}. For asymptotic profiles of solutions to the corresponding abstract form of \eqref{Eq.Linear.SDW} were derived in  \cite{IkehataTodorovaYordanov2013}. Moreover, \cite{DAbbiccoReissig2014} deduced the $L^2-L^2$ estimate with additional $L^1$ regularity for \eqref{Eq.Linear.SDW}, which provides a useful tool to prove global existence result for the semilinear Cauchy problem. Furthermore, \cite{DAbbiccoReissig2014} considered the corresponding semilinear Cauchy problem to \eqref{Eq.Linear.SDW} with power nonlinearity, i.e.,
\begin{equation}\label{Eq.Semilinear.SDW}
\begin{cases}
u_{tt}-\Delta u-\Delta u_t=|u|^p,&x\in\mathbb{R}^n,\,t>0,\\
u(0,x)=  u_0(x),\,u_t(0,x)=  u_1(x),&x\in \mathbb{R}^n,
\end{cases}
\end{equation}
where $p>1$. They proved global (in time) existence of small data solutions (GESDS) for $n\geqslant2$ if $p\in[2,n/(n-4)]$ and $p>1+3/(n-1)$. On the contrary, by applying the test function method, the result for nonexistence of global (in time) solutions has been proved providing that $1<p\leqslant 1+2/(n-1)$, for example, Theorem 4.2 in \cite{D'AmbrosioLucente2003}.

We now consider strongly damped wave equation in the exterior domain.  To the best of the authors' knowledge, there exist few results on \eqref{eq1}. We refer to  \cite{Ikehata} for decay estimates to the linearized problem. Additionally, \cite{IkehataInoue} proved GESDS to the two-dimensional semilinear problem with power nonlinearity $f(u,u_t)=|u|^p$ when $p>6$. Recently, the blow-up of solutions to semilinear wave equation with strong damping and power nonlinearity $f(u,u_t)=|u|^p$ has been obtained by \cite{Fino}. Nevertheless, so far the study of semilinear strongly damped wave equations with power nonlinear term of the derivative-type $|u_t|^q$ or the mixed-type $|u|^p+|u_t|^q$ in an exterior domain are still unknown. In this paper, we will give the answer by studying the local (in time) existence of mild solution and nonexistence of global (in time) solutions to \eqref{eq1}. However, the local (in time) existence for strongly damped waves with nonlienarity $f(u,u_t)$ satisfying \eqref{Nonlinearity} is different with the case when $f(u,u_t)=|u|^p$ studied by \cite{Ikehata} or \cite{Fino}. For example, we need the estimate \eqref{eq2.8} below, which was not mentioned before.

Before stating our main results on blow-up, the next proposition on local (in time) well-posedness is needed. For the proof of this proposition, we refer to \cite{Yuta} by an appropriate modification of the energy space and some estimations. For simplicity, we present all in details with all required modifications in Section \ref{sec3}.
\begin{prop}[Local existence of mild solution]\label{prop1} Let us assume the exponents of power nonlinear terms $f(u,u_t)$ satisfies
	\begin{equation}\label{assumption1}
	\begin{cases}
	1<p,q<\infty &\mbox{for}\; n=1,2,\\
	1<p,q\leqslant \frac{n}{n-2}&\mbox{for}\; n\geqslant 3,
	\end{cases}
	\end{equation} 
	and initial data fulfill
	\begin{align}\label{Condition.Initial.Data}
	(u_0,u_1)\in H^1_0(\Omega)\times H^1(\Omega).
	\end{align}
	Then, there exists a maximal existence time $0<T_{\max}\leqslant\infty$ such that there is a uniquely mild solution
	\begin{align*}
	u\in \mathcal{C}\left(\left[0,T_{\max}\right),H^1_0(\Omega)\right)\cap \mathcal{C}^1\left(\left[0,T_{\max}\right), H^1(\Omega)\right)
	\end{align*}
	to \eqref{eq1}.  Furthermore, the following statement holds:
	\begin{equation}\label{alternative}
	\mbox{either}\,\, T_{\max}=\infty \quad\mbox{or else}\quad T_{\max}<\infty \;\mbox{and}\; \|u(t,\cdot)\|_{H^1}+\|u_t(t,\cdot)\|_{H^1}\rightarrow\infty\;\;\mbox{as}\;\; t\rightarrow T_{\max}.
	\end{equation}
\end{prop}
\begin{rem}
	In Proposition \ref{prop1}, we say that $u$ is a global (in time) solution of $\eqref{eq1}$ if $T_{\max}=\infty$, while in the case of $T_{\max}<\infty$, we say that the solution $u$ blows up in finite time.
\end{rem}

We now state our main results on blow-up of solutions, whose proofs are based on the test function method (see, for example, \cite{Zhang,Finogeorgiev,FinoKarch,Finokirane,PM1}).
\begin{thm}[Blow-up for nonlinearity of derivative-type]\label{a=0}
	Let us assume $(u_0,u_1)\in \left(H^1_0(\Omega)\cap L^1(\Omega)\right)\times H^1(\Omega)$ and $\phi_0u_1\in L^1(\Omega)$ such that
	\begin{align*}
	\int\nolimits_{\Omega}\phi_0(x)u_1(x)\,\mathrm{d}x>0,
	\end{align*}
	where $\phi_0(x)$ is defined in Lemmas \ref{lemma1}, \ref{lemma4} and \ref{lemma2}.\\
	If the exponent fulfills
	\begin{equation*}
	\begin{cases}
	1<q\leqslant 1+\frac{1}{1+\sqrt{5}}& \mbox{for}\; n=1,\\
	1<q<1+\frac{1}{2}& \mbox{for}\;n=2,\\
	1<q\leqslant 1+\frac{1}{n}& \mbox{for}\;n\geqslant 3,
	\end{cases}
	\end{equation*}
	then the solution to \eqref{eq1} with nonlinear term $f(u,u_t)=|u_t|^q$ blows up in finite time.
\end{thm}
\begin{thm}[Blow-up for nonlinearity of mixed-type]\label{a>0} Let us assume $(u_0,u_1)\in \left(H^1_0(\Omega)\cap L^1(\Omega)\right)\times H^1(\Omega)$ and $\phi_0u_1\in L^1(\Omega)$ such that
	\begin{align*}
	\int\nolimits_{\Omega}\phi_0(x)u_1(x)\,\mathrm{d}x>0,
	\end{align*}
	where $\phi_0(x)$ is defined in Lemmas \ref{lemma1}, \ref{lemma4} and \ref{lemma2}.\\
	In the case when $n=1$, if one of the following condition is fulfilled:
	\begin{align*}
	\begin{cases}
	1<p\leqslant 1+\alpha_1\approx 2.28 \quad\hbox{and}\quad 1<q,\\
	1<p\leqslant \frac{2\alpha+1}{2\alpha-1}\quad\hbox{and}\quad 1<q\leqslant \frac{\alpha+1}{2}\quad\hbox{for any}\,\alpha<\alpha_1,\\
	1<p\leqslant \alpha+1\quad\hbox{and}\quad 1<q\leqslant \frac{2\alpha+1}{2\alpha}\quad\hbox{for any}\,\alpha_1<\alpha<\alpha_2,\\
	1<p \quad\hbox{and }\quad 1<q\leqslant \frac{1+\alpha_2}{2}\approx 1.3,
	\end{cases}
	\end{align*}
	where $\alpha_1=(1+\sqrt{17}\,)/4$ is the positive root of $2\alpha^2-\alpha-2=0$, and $\alpha_2=(1+\sqrt{5}\,)/2$ is the positive root of $\alpha^2-\alpha-1=0$; in the case when $n\geqslant 2$, if one of the following condition is fulfilled:
	\begin{align*}
	\begin{cases}
	1<p<3\quad\mbox{or}\quad 1<q< 1+\frac{1}{2}& \mbox{for}\;n=2,\\
	1<p\leqslant 1+\frac{2}{n-1}\quad\mbox{or}\quad 1<q\leqslant 1+\frac{1}{n}& \mbox{for}\;n\geqslant 3,
	\end{cases}
	\end{align*}	
	then the solution to \eqref{eq1} with nonlinear term $f(u,u_t)=|u|^p+|u_t|^q$ blows up in finite time.
\end{thm}

The remaining part of paper is organized as follows. In Section \ref{sec2}, we introduce several preliminaries including well-posedness of the corresponding linearized inhomogeneous problem to \eqref{eq1} and some properties of harmonic functions to be used later. In Section \ref{sec3}, we prove local (in time) existence of mild solution (cf. Proposition \ref{prop1}). Section \ref{sec4} contains the proofs of the blow-up results (cf. Theorems \ref{a=0} and \ref{a>0}).

\medskip
\noindent\textbf{Notation.} We now give some notations to be used in this paper. We denote by$\|\cdot\|_{H^1_0}=\|\cdot\|_{H^1}=\|\cdot\|_{L^2}+\|\nabla \cdot\|_{L^2}$ the usual $H^1_0(\Omega)$-norm.  Moreover, we denote denote by $W^{1,r}$ with $1<r<\infty$ Bessel potential spaces.


\section{Preliminaries}\label{sec2}
In this section, we give some preliminary properties that will be used in the proof of local (in time) existence of mild solution and blow-up results in the remaining sections.
\subsection{Linear homogeneous strongly damped wave equation}
Let us consider the following linear homogeneous wave equation with strong damping:
\begin{equation}\label{eq2.1}
\begin{cases}
u_{tt}-\Delta u -\Delta u_t =0, &x\in {\Omega},\,t>0,\\
u(0,x)= u_0(x),\;u_t(0,x)=  u_1(x),&x\in {\Omega},\\
u=0,&x\in {\partial\Omega},\,t>0,
\end{cases}
\end{equation}
which is the corresponding linearized model to \eqref{eq1}.

To begin with, we give the definition of a strong solution to \eqref{eq2.1}. 
\begin{defn}[Strong solution] Let $(u_0,u_1)\in \left(H^2(\Omega)\cap H^1_0(\Omega)\right)^2$. A function $u$ is said to be a strong solution to \eqref{eq2.1} if
	\begin{align*}
	u\in \mathcal{C}^{1}\left([0,\infty),H^2(\Omega)\cap H^1_0(\Omega)\right)\cap \mathcal{C}^{2}\left([0,\infty),L^2(\Omega)\right)
	\end{align*}
	and $u$ has initial data $u(0,x)=u_0(x)$, $u_t(0,x)=u_1(x)$ and satisfies
	the equation in \eqref{eq2.1} in the sense of $L^2(\Omega)$.
\end{defn}
\begin{prop}\label{prop2.1}
	For each $(u_0,u_1)\in \left(H^2(\Omega)\cap H^1_0(\Omega)\right)^2$, there exists a unique strong solution $u$ to \eqref{eq2.1} that satisfies the following energy estimates:
	\begin{align}
	\|(u_t,\nabla u)(t,\cdot)\|_{L^2}&\leqslant \|(u_1,\nabla u_0)\|_{L^2\times L^2} \quad \mbox{for any}\;t\geqslant 0,\label{eq2.2}
	\end{align}
	\begin{align}
	\|u(t,\cdot)\|_{L^2}&\leqslant \|u_0\|_{L^2}+T\|(u_1,\nabla
	u_0)\|_{L^2\times
		L^2}\quad \mbox{for any $T>0$ and all $0\leqslant t\leqslant T$}.\label{eq2.3}
	\end{align}
\end{prop}
\begin{proof}
	The existence of the strong solution is done by Theorem 3.1 in \cite{KPS}. The energy estimates \eqref{eq2.2} and \eqref{eq2.3} can be deduced easily from Proposition 1.1 in\cite{Ikehata} or directly from Proposition 2 in \cite{Fino}.
\end{proof}

Let us denote the operator $R(t)$ such that
\begin{align*}
R(t):\,(u_0,u_1)\in
\left(H^2(\Omega)\cap H^1_0(\Omega)\right)^2\,\longrightarrow\,u(t,\cdot)\in
H^2(\Omega)\cap H^1_0(\Omega)
\end{align*}
at the time $t \geqslant 0$. In other words, the solution $u$ of \eqref{eq2.1}
can be given by $u(t,x)=R(t)(u_0,u_1)(x)$. 
\begin{rem}\label{rmk2.1}
	From Proposition \ref{prop2.1}, the operator $R(t)$ can be
	extended uniquely such that 
	\begin{align}\label{Extend.Operator}
	R(t):\,H_0^1(\Omega)\times
	L^2(\Omega)\,\longrightarrow\, \mathcal{C}\left([0,\infty),H_0^1(\Omega)\right)\cap
	\mathcal{C}^1\left([0,\infty),L^2(\Omega)\right).
	\end{align}
	Indeed, for any fixed $T>0$, due to the energy estimates \eqref{eq2.2} and \eqref{eq2.3}, the following estimate:
	\begin{align*}
	\|R(t)(u_0,u_1)(\cdot)\|_{H_0^1}+\|\partial_t(R(t)(u_0,u_1))(\cdot)\|_{L^2}\leqslant
	C(1+T)\,\|(u_0,u_1)\|_{H_0^1\times L^2}
	\end{align*}
	holds for all $0\leqslant t\leqslant T$. It follows that the operator $R(t)$ can be extended uniquely to an operator such as \eqref{Extend.Operator}. Since $T$ is arbitrary, we conclude the desired extension.
\end{rem}

\subsection{Linear inhomogeneous strongly damped wave equation}
Let us now consider the linear inhomogeneous wave equation with strong damping, namely,
\begin{equation}\label{eq2.4}
\begin{cases}
u_{tt}-\Delta
u -\Delta u_t =F(t,x), &x\in {\Omega},\,t>0,\\
u(0,x)=  u_0(x),\;u_t(0,x)=  u_1(x),&x\in {\Omega},\\
u=0,&x\in {\partial\Omega},\,t>0.
\end{cases}
\end{equation}

At this time, the definition of a strong solution to \eqref{eq2.4} can be shown by the following.
\begin{defn}[Strong solution] Let  $(u_0,u_1)\in \left(H^2(\Omega)\cap H^1_0(\Omega)\right)^2$ and $F\in \mathcal{C}\left([0,\infty),L^2(\Omega)\right)$. A function $u$ is said to be a strong solution of \eqref{eq2.4} if
	\begin{align*}
	u\in \mathcal{C}^{1}\left([0,\infty),H^2(\Omega)\cap H^1_0(\Omega)\right)\cap \mathcal{C}^{2}\left([0,\infty),L^2(\Omega)\right),
	\end{align*}
	and $u$ has initial data $u(0,x)=u_0(x)$, $u_t(0,x)=u_1(x)$ and satisfies
	the equation in \eqref{eq2.4} in the sense of $L^2(\Omega)$.
\end{defn}
\begin{prop}[Theorem 3.1 in \cite{KPS}]\label{prop2.2} Let $(u_0,u_1)\in \left(H^2(\Omega)\cap H^1_0(\Omega)\right)^2$ and $F\in \mathcal{C}^1\left([0,\infty),L^2(\Omega)\right)$. Then, there exists a unique strong solution to \eqref{eq2.4}.
\end{prop}

Next, let us define a mild solution and a weak solution to inhomogeneous problem, one by one.
\begin{defn}[Mild solution] Let $(u_0,u_1)\in \left(H^1_0(\Omega)\times L^2(\Omega)\right)^2$ and $F\in \mathcal{C}\left([0,\infty),L^2(\Omega)\right)$. A function $u$ is said to be a mild solution to \eqref{eq2.4} if 
	\begin{align*}
	u\in \mathcal{C}\left([0,\infty),H^1_0(\Omega)\right)\cap \mathcal{C}^{1}\left([0,\infty),H^1(\Omega)\right),
	\end{align*}
	and $u$ has initial data $u(0,x)=u_0(x)$, $u_t(0,x)=u_1(x)$ and satisfies the integral equation
	\begin{align}\label{eq2.5}
	u(t,x)=R(t)(u_0,u_1)(x)+\int\nolimits_0^tS(t-s)F(s,x)\,\mathrm{d}s
	\end{align}
	in the sense of $H^1(\Omega)$, where $S(t)g(x):=R(t)(0,g)(x)$ for all $g\in L^2(\Omega)$.
\end{defn}
\begin{rem}
	It is easy to check that $S(t)(-\Delta u_0+u_1)(x)+\partial_t S(t)u_0(x)$ is a strong solution of \eqref{eq2.1}. It follows by the uniqueness that
	\begin{align*}
	R(t)(u_0,u_1)(x)=S(t)(-\Delta u_0(x)+u_1(x))+\partial_t S(t)u_0(x).
	\end{align*}
\end{rem}
\begin{defn}[Weak solution] Let $T>0$, $(u_0,u_1)\in L_{\mathrm{loc}}^1(\Omega)\times L_{\mathrm{loc}}^1(\Omega)$ and  $F\in L^1\left((0,T),L_{\mathrm{loc}}^1(\Omega)\right)$. A function $u$ is said to be a weak solution to \eqref{eq2.4} if
	\begin{align*}
	u\in L^1\left((0,T),L_{\mathrm{loc}}^1(\Omega)\right),
	\end{align*}
	and $u$ has initial data $u(0,x)=u_0(x)$, $u_t(0,x)=u_1(x)$ and satisfies the relation
	\begin{align}\label{eq2.6}
	&\int\nolimits_0^T\int\nolimits_{\Omega}F(t,x)\varphi(t,x)\,\mathrm{d}x\,\mathrm{d}t+\int\nolimits_{\Omega}u_1(x)\varphi(0,x)\,\mathrm{d}x-\int\nolimits_{\Omega}u_0(x)\Delta\varphi(0,x)\,\mathrm{d}x-\int\nolimits_{\Omega}u_0(x)\varphi_t(0,x)\,\mathrm{d}x\notag\\
	&=\int\nolimits_0^T\int\nolimits_{\Omega}u(0,x)\varphi_{tt}(0,x)\,\mathrm{d}x\,\mathrm{d}t+\int\nolimits_0^T\int\nolimits_{\Omega}u(t,x)\Delta\varphi_t(t,x)\,\mathrm{d}x\,\mathrm{d}t-\int\nolimits_0^T\int\nolimits_{\Omega}u(t,x)\Delta\varphi(t,x)\,\mathrm{d}x\,\mathrm{d}t
	\end{align}
	for all compactly supported test function $\varphi\in \mathcal{C}^2([0,T]\times\Omega)$ such that $\varphi(T,x)=0$ and $\varphi_t(T,x)=0$.
\end{defn}
\begin{prop}\label{prop2.3} Let $(u_0,u_1)\in \left(H^2(\Omega)\cap H^1_0(\Omega)\right)^2$ and $$F\in \mathcal{C}\left([0,\infty),H^2(\Omega)\cap H^1_0(\Omega)\right)\cap \mathcal{C}^1\left([0,\infty),L^2(\Omega)\right)$$ and $u$ be a strong solution of \eqref{eq2.4}. Then $u$ also is a mild solution, and satisfies the following energy estimates:
	\begin{align}
	&\|(u_t,\nabla u)(t,\cdot)\|_{L^2\times L^2}\leqslant C\|(u_1,\nabla
	u_0)\|_{L^2\times L^2}+C\int\nolimits_0^t\|F(s,\cdot)\|_{L^2}\,\mathrm{d}s,\label{eq2.7}\\
	&\|u(t,\cdot)\|_{L^2}\leqslant \|u_0\|_{L^2}+C\int\nolimits_0^t\left(\|(u_1,\nabla
	u_0)\|_{L^2\times
		L^2}+\int\nolimits_0^s\|F(\tau,\cdot)\|_{L^2}\,\mathrm{d}\tau\right)\mathrm{d}s,\label{eq2.8}\\
	&\|\nabla u_t(t,\cdot)\|^2_{L^2}\leqslant  C\|(\nabla
	u_0,u_1)\|^2_{L^2\times H^1}+C\int\nolimits_0^t\|F(s,\cdot)\|_{L^2}^2\,\mathrm{d}s +C\|\nabla u(t,\cdot)\|^2_{L^2}\notag\\
	&\qquad\qquad\qquad\,\,\,\,+C\int\nolimits_0^t \|F(s,\cdot)\|_{L^2}\|u_s(s,\cdot)\|_{L^2}\,\mathrm{d}s.\label{eq2.16}
	\end{align}
	
\end{prop}
\begin{proof}
	Let us define
	\begin{align*}
	\tilde{u}(t,x):=R(t)(u_0,u_1)(x)+\int\nolimits_0^tS(t-s)F(s,x)\,\mathrm{d}s.
	\end{align*}
	According to the assumptions on initial data and the right-hand sides, the Proposition \ref{prop2.1} leads that $\tilde{u}$ is a strong solution of \eqref{eq2.4}. Hence, by the uniqueness we claim $u=\tilde{u}$, i.e., $u$ is a mild solution.\\ 
	Next, we start to prove estimate \eqref{eq2.7}. Multiplying \eqref{eq2.4} by $u_t$ and integrating the result over $\Omega$  one has
	\begin{align*}
	\frac{1}{2}\frac{\mathrm{d}}{\mathrm{d}t}\int\nolimits_\Omega \left(|u_t(t,x)|^2+|\nabla u(t,x)|^2\right)\mathrm{d}x+\int\nolimits_\Omega |\nabla u_t(t,x)|^2\,\mathrm{d}x=\int\nolimits_\Omega F(t,x)u_t(t,x)\,\mathrm{d}x,
	\end{align*}
	where the divergence theorem was applied together with the boundary condition. Integrating the above equality over $[0,t]$, and using the Cauchy-Schwarz inequality, we deduce that
	\begin{align*}
	&\frac{1}{2}\int\nolimits_\Omega \left(|u_t(t,x)|^2+|\nabla u(t,x)|^2\right)\mathrm{d}x+\int\nolimits_0^t\int\nolimits_\Omega |\nabla u_s(s,x)|^2\,\mathrm{d}x\,\mathrm{d}s\\
	&\leqslant \frac{1}{2}\|(u_1,\nabla u_0)\|_{L^2\times L^2}^2+\int\nolimits_0^t \|F(s,\cdot)\|_{L^2}\|u_s(s,\cdot)\|_{L^2}\,\mathrm{d}s.
	\end{align*} 
	It follows that
	\begin{align}
	\int\nolimits_0^t\int\nolimits_\Omega |\nabla u_s(s,x)|^2\,\mathrm{d}x\,\mathrm{d}s&\leqslant \frac{1}{2}\|(u_1,\nabla u_0)\|_{L^2\times L^2}^2+\int\nolimits_0^t \|F(s,\cdot)\|_{L^2}\|u_s(s,\cdot)\|_{L^2}\,\mathrm{d}s,\label{eq2.13}
	\end{align}
	and
	\begin{align}
	&\frac{1}{2}\int\nolimits_\Omega \left(|u_t(t,x)|^2+|\nabla u(t,x)|^2\right)\mathrm{d}x\notag\\
	&\leqslant \frac{1}{2}\|(u_1,\nabla u_0)\|_{L^2\times L^2}^2+\int\nolimits_0^t \|F(s,\cdot)\|_{L^2}\left(\|u_s(s,\cdot)\|_{L^2}^2+\|\nabla u(s,\cdot)\|^2_{L^2}\right)^{1/2}\mathrm{d}s.\label{eq2.14}
	\end{align}
	Applying Gr\"onwall's Lemma (see e.g. Lemma 9.12 in \cite{Yuta}) to \eqref{eq2.14}, we conclude our desired estimates.
	%
	
	To prove \eqref{eq2.16}, multiplying \eqref{eq2.4} by $u_{tt}$ and integrating it over $\Omega$, we immediately derive
	\begin{align*}\int\nolimits_\Omega |u_{tt}(t,x)|^2\,\mathrm{d}x-\int\nolimits_\Omega \Delta u(t,x)u_{tt}(t,x)\,\mathrm{d}x-\int\nolimits_\Omega \Delta u_t(t,x)u_{tt}(t,x)\,\mathrm{d}x=\int\nolimits_\Omega F(t,x)u_{tt}(t,x)\,\mathrm{d}x.
	\end{align*}
	By the boundary condition, Young's inequality and the divergence theorem, the following estimate holds:
	\begin{align*}
	&\frac{1}{2}\frac{\mathrm{d}}{\mathrm{d}t}\int\nolimits_\Omega |\nabla u_t(t,x)|^2\,\mathrm{d}x+\int\nolimits_\Omega |u_{tt}(t,x)|^2\,\mathrm{d}x+\int\nolimits_\Omega\nabla u(t,x)\cdot \nabla u_{tt}(t,x)\,\mathrm{d}x\\
	&= \int\nolimits_\Omega F(t,x)u_{tt}(t,x)\,\mathrm{d}x\leqslant \frac{1}{2}\int\nolimits_\Omega |F(t,x)|^2\,\mathrm{d}x+ \frac{1}{2}\int\nolimits_\Omega |u_{tt}(t,x)|^2\,\mathrm{d}x,
	\end{align*}
	which implies
	\begin{align*}
	\frac{1}{2}\frac{\mathrm{d}}{\mathrm{d}t}\int\nolimits_\Omega |\nabla u_t(t,x)|^2\,\mathrm{d}x+\int\nolimits_\Omega\nabla u(t,x)\cdot (\nabla u_t(t,x))_t\,\mathrm{d}x\leqslant \frac{1}{2}\int\nolimits_\Omega |F(t,x)|^2\,\mathrm{d}x.
	\end{align*}
	Integrating over $[0,t]$, we conclude 
	\begin{align*}
	\frac{1}{2}\int\nolimits_\Omega |\nabla u_t(t,x)|^2\,\mathrm{d}x+\int\nolimits_0^t\int\nolimits_\Omega\nabla u(t,x)\cdot (\nabla u_s(s,x))_s\,\mathrm{d}x\,\mathrm{d}s\leqslant \frac{1}{2}\int\nolimits_0^t\int\nolimits_\Omega |F(s,x)|^2\,\mathrm{d}x\,\mathrm{d}s+ \frac{1}{2}\|\nabla u_1\|^2_{L^2}.
	\end{align*}
	Then, we derive
	\begin{align*}
	\frac{1}{2}\int\nolimits_\Omega |\nabla u_t(t,x)|^2\,\mathrm{d}x&\leqslant \frac{1}{2}\|\nabla u_1\|_{L^2}^2+\int\nolimits_0^t\int\nolimits_\Omega |\nabla u_s(s,x)|^2\,\mathrm{d}x\,\mathrm{d}s-\int\nolimits_\Omega \nabla u(t,x)\cdot\nabla u_t(t,x)\,\mathrm{d}x\\
	&\quad+\int\nolimits_\Omega \nabla u_0(x)\cdot \nabla u_1(x)\,\mathrm{d}x+\frac{1}{2}\int\nolimits_0^t\int\nolimits_\Omega |F(s,x)|^2\,\mathrm{d}x\,\mathrm{d}s\\
	&\leqslant\|\nabla u_1\|_{L^2}^2+\frac{1}{2}\|\nabla u_0\|_{L^2}^2+\int\nolimits_0^t\int\nolimits_\Omega |\nabla u_s(s,x)|^2\,\mathrm{d}x\,\mathrm{d}s+\frac{1}{4}\int\nolimits_\Omega |\nabla u_t(t,x)|^2\,\mathrm{d}x\\
	&\quad+\int\nolimits_\Omega |\nabla u(t,x)|^2\,\mathrm{d}x+\frac{1}{2}\int\nolimits_0^t\int\nolimits_\Omega |F(s,x)|^2\,\mathrm{d}x\,\mathrm{d}s,
	\end{align*}
	where we have used Young's inequalities such that $AB\leqslant\frac{1}{4}A^2+B^2$ and $AB\leqslant\frac{1}{2}A^2+\frac{1}{2}B^2$, for all $A,B\geqslant 0$. So,
	it implies together with \eqref{eq2.7} and \eqref{eq2.13} the desired estimate.
	
	Finally, we begin to prove \eqref{eq2.8}. Let $T>0$, for all $0\leqslant t\leqslant T$ the integral formula shows
	\begin{align*}
	u(t,x)=u_0(x)+\int\nolimits_0^tu_s(s,x)\,\mathrm{d}s.
	\end{align*}
	By using our derived result \eqref{eq2.7}, we achieve our aim.
	This completes the proof.
\end{proof}
\begin{prop}\label{prop2.4}
	Let $(u_0,u_1)\in H_0^1(\Omega)\times H^1(\Omega)$ and $F\in \ml{C}\left([0,\infty),L^2(\Omega)\right)$. Then,
	there exists a unique mild solution $u$ to \eqref{eq2.4}. Moreover,
	the mild solution $u$ satisfies the estimates \eqref{eq2.7}-\eqref{eq2.16}.
\end{prop}
\begin{proof}
	Let us prove existence first. Let $T_0>0$ be an arbitrary number. By the density argument, there exist sequences 
	\begin{align*}
	\left\{\left(u_0^{(j)},u_1^{(j)}\right)\right\}_{j=1}^\infty\subseteq
	\left(H^2(\Omega)\cap H^1_0(\Omega)\right)^2,
	\end{align*}
	\begin{align*}
	\left\{F^{(j)}\right\}_{j=1}^\infty\subseteq
	\ml{C}\left([0,T_0],H^2(\Omega)\cap H^1_0(\Omega)\right)\cap \ml{C}^1\left([0,T_0],L^2(\Omega)\right)
	\end{align*}
	such that
	\begin{align*}
	\lim_{j\rightarrow\infty}\left(u_0^{(j)},u_1^{(j)}\right)=(u_0,u_1)\;\;\mbox{in}\;\; H^1_0(\Omega)\times H^1(\Omega),\quad \lim_{j\rightarrow\infty}F^{(j)}=F\;\;\mbox{in}\;\; \ml{C}\left([0,T_0],L^2(\Omega)\right).
	\end{align*}
	Using Proposition \ref{prop2.2}, let $u^{(j)}$ be the strong solution of the linear inhomogeneous
	equation \eqref{eq2.4} with initial data $\left(u_0^{(j)},u_1^{(j)}\right)$
	and the inhomogeneous term $F^{(j)}(t,x)$. Then, the difference $u^{(j)}-u^{(k)}$ with $j,k\geqslant 1$, is a strong solution of the initial value problem 
	\begin{align*}
	\begin{cases}
	u_{tt}-\Delta
	u -\Delta u_t =F^{(j)}(t,x)-F^{(k)}(t,x), &x\in {\Omega},\,t>0,\\
	u(0,x)=  u^{(j)}_0(x)-u^{(k)}_0(x),\;u_t(0,x)=  u_1^{(j)}(x)-u_1^{(k)}(x),&x\in {\Omega},\\
	u=0,&x\in\partial\Omega,\,t>0.
	\end{cases}
	\end{align*}
	
	Applying Proposition \ref{prop2.3} to $u^{(j)}-u^{(k)}$, we have
	\begin{align*}
	\left\|\left(\partial_t\left(u^{(j)}-u^{(k)}\right),\nabla \left(u^{(j)}-u^{(k)}\right)\right)(t,\cdot)\right\|_{L^2\times L^2}&\leqslant C\left\|\left(u_1^{(j)}-u_1^{(k)},\nabla
	\left(u_0^{(j)}-u_0^{(k)}\right)\right)\right\|_{L^2\times L^2}\\
	&\quad+C\,T_0\sup_{s\in[0,T_0]}\left\|\left(F^{(j)}-F^{(k)}\right)(s,\cdot)\right\|_{L^2},
	\end{align*}
	\begin{align*}
	\left\|\left(u^{(j)}-u^{(k)}\right)(t,\cdot)\right\|_{L^2}&\leqslant \left\|u_0^{(j)}-u_0^{(k)}\right\|_{L^2}+C\,T_0\left\|\left(u_1^{(j)}-u_1^{(k)},\nabla
	\left(u_0^{(j)}-u_0^{(k)}\right)\right)\right\|_{L^2\times
		L^2}\\
	&\quad+C\,T_0^2\sup_{s\in[0,T_0]}\left\|\left(F^{(j)}-F^{(k)}\right)(s,\cdot)\right\|_{L^2},
	\end{align*}
	and
	\begin{align*}
	&\left\|\nabla\partial_t\left(u^{(j)}-u^{(k)}\right)(t,\cdot)\right\|_{L^2}\\
	&\leqslant C\left\|\left(\nabla\left(
	u_0^{(j)}-u_0^{(k)}\right),u_1^{(j)}-u_1^{(k)}\right)\right\|^2_{L^2\times H^1}+C\,T_0\sup_{s\in[0,T_0]}\left\|\left(F^{(j)}-F^{(k)}\right)(s,\cdot)\right\|^2_{L^2} \\
	&\quad+C\left\|\nabla  \left(u^{(j)}-u^{(k)}\right)(t,\cdot)\right\|^2_{L^2}+C\sup_{s\in[0,T_0]}\left\|\left(F^{(j)}-F^{(k)}\right)(s,\cdot)\right\|_{L^2}\int\nolimits_0^t \left\|\partial_s\left(u^{(j)}-u^{(k)}\right)(s,\cdot)\right\|_{L^2}\,\mathrm{d}s.
	\end{align*}
	This shows that $\left\{u^{(j)}\right\}_{j=1}^\infty$ is a Cauchy sequence in the Banach space $$\ml{C}\left([0,T_0],H^1_0(\Omega)\right)\cap \ml{C}^1\left([0,T_0],H^1(\Omega)\right).$$ Therefore, we can define the limit 
	\begin{equation}\label{eq2.9}
	\lim_{j\rightarrow\infty}u^{(j)}=u\in \ml{C}\left([0,\infty),H^1_0(\Omega)\right)\cap \ml{C}^1\left([0,\infty),H^1(\Omega)\right),
	\end{equation}
	since $T_0>0$ is arbitrary. Applying again Proposition \ref{prop2.3} to $u^{(j)}$, it follows that $u^{(j)}$ satisfies the integral equation
	\begin{align*}
	u^{(j)}(t,x)=R(t)\left(u^{(j)}_0,u^{(j)}_1\right)(x)+\int\nolimits_0^tS(t-s)F^{(j)}(s,x)\,\mathrm{d}s.
	\end{align*}
	According to Remark \ref{rmk2.1}, the operators $R(t)$ and $S(t)$ can be extended uniquely to the operators defined on $H_0^1(\Omega)\times L^2(\Omega)$ and $L^2(\Omega)$, respectively. Letting $j\rightarrow\infty$, one may obtain
	\begin{align*}
	u(t,x)=R(t)(u_0,u_1)(x)+\int\nolimits_0^tS(t-s)F(s,x)\,\mathrm{d}s,
	\end{align*}
	which indicates that $u$ is a mild solution of \eqref{eq2.4}.
	
	To prove uniqueness, we find that if two functions $u$ and $v$ satisfy the integral equation \eqref{eq2.5}, then we immediately have $u=v$.
	
	Concerning energy estimates, by Proposition \ref{prop2.3}, each strong solution $u^{(j)}$ constructed above satisfies the estimates \eqref{eq2.7}-\eqref{eq2.16} with $u_0^{(j)},u_1^{(j)},F^{(j)}$. By letting $j\rightarrow\infty$ and using \eqref{eq2.9}, the same estimates hold for the mild solution $u$. The proof is complete.
\end{proof}

\subsection{Harmonic functions}
In this subsection, we give some harmonic function that will be used in the proof of Theorems \ref{a=0} and \ref{a>0}.
\begin{lem}\label{lemma1}
	There exists a function $\phi_0(x)\in C^2(\Omega)\cap C(\overline{\Omega})$ for $n\geqslant 3$ satisfying the boundary value problem
	\begin{equation}\label{function1}
	\begin{cases}
	\Delta \phi_0(x)=0, \;&x\in\Omega,\\
	\phi_0(x)=0,\; &x\in\partial\Omega,\\
	\phi_0(x)\rightarrow 1,\,&|x|\rightarrow\infty.
	\end{cases}
	\end{equation}
	Moreover, the function $\phi_0(x)$ satisfies $0<\phi_0(x)<1$ for all $x\in\Omega$, and $\phi_0(x)\geqslant C$ for all $|x|\gg1$. Furthermore, for all $|x|\gg1$ we have $|\nabla \phi_0(x)|\leqslant C|x|^{1-n}$.
\end{lem}
\begin{proof}
	From Lemma 2.2 in \cite{ZhouHan} there exists a regular solution $\phi_0$ of \eqref{function1} such that $0<\phi_0(x)<1$, for all $x\in\Omega$. To obtain the last two properties of $\phi_0$, it is easy to see that since $\mathcal{O}$ is bounded, there exist $r_2>r_1>0$ such that $B(r_1)\subseteq\mathcal{O}\subseteq B(r_2)$. By the maximum principle we conclude that $\phi_1(x)\leqslant \phi_0(x)\leqslant \phi_2(x)$ in $\Omega$, where $\phi_1(x)$ and $\phi_2(x)$ are, respectively, the solution of \eqref{function1} on $\mathbb{R}^n\setminus B(r_1)$ and $\mathbb{R}^n\setminus B(r_2)$. We remember that $\phi_i(x)=r_i^{2-n}-|x|^{2-n}$ for $i=1,2$. Moreover, the standard elliptic theory implies that $|\nabla \phi_0(x)|\sim |\nabla \phi_{i}(x)|$ for $i=1,2$. As $|\phi_{1}(x)|\geqslant C$ and $|\nabla\phi_{i}(x)|\leqslant C|x|^{1-n}$ when $|x|\gg1$, this completes the proof.
\end{proof}

Similarly, we have the following Lemmas in one and two dimensions, respectively.
\begin{lem}[Lemma 2.5 in \cite{Han}]\label{lemma4}
	There exists a function $\phi_0(x)\in \ml{C}^2(\Omega)\cap \ml{C}(\overline{\Omega})$ for $n=2$ satisfying the boundary value problem
	\begin{equation}\label{function2}
	\begin{cases}
	\Delta \phi_0(x)=0, \;&x\in\Omega,\\
	\phi_0(x)=0,\; &x\in\partial\Omega,\\
	\phi_0(x)\rightarrow +\infty,\,&|x|\rightarrow\infty\;\;\mbox{and}\;\phi_0(x)\;\mbox{increases at the rate of}\;\ln (|x|).
	\end{cases}
	\end{equation}
	Moreover, the function $\phi_0(x)$ satisfies $0<\phi_0(x)\leqslant C\ln (|x|)$ for all $x\in\Omega$, and $\phi_0(x)\geqslant C$ for all $|x|\gg1$. Furthermore, for all $|x|\gg1$, $|\nabla \phi_0(x)|\leqslant C|x|^{-1}$.
\end{lem}
\begin{lem}[Lemma 2.2 in \cite{Han2}]\label{lemma2}
	There exists a function $\phi_0(x)\in \ml{C}^2([0,\infty))$ for $x\geqslant0$ satisfying the boundary value problem
	\begin{equation}\label{function3}
	\begin{cases}
	\Delta \phi_0(x)=0, \;&x>0,\\
	\phi_0(x)=0,\; &x=0,\\
	\phi_0(x)\rightarrow +\infty,\,&|x|\rightarrow\infty\;\;\mbox{and}\;\phi_0(x)\;\mbox{increases at the rate of linear function $x$}.
	\end{cases}
	\end{equation}
	Moreover, the function $\phi_0(x)$ satisfies that there exist two positive constants $C_1$ and $C_2$ such that, for all $x>0$, we have $C_1x\leqslant \phi_0(x)\leqslant C_2x$. In other words, we can take $\phi_0(x)=Cx$.
\end{lem}


\section{Local existence of mild solutions}\label{sec3}
In this section, we will prove the local (in time) existence of mild solution (Proposition \ref{prop1}). We start by giving the definition of the mild and weak solution of \eqref{eq1}. Clearly, for a nonlinear equation it is not always true that the solution exists globally in-time. Therefore, we consider the solution defined on an interval $[0,T)$ for $T>0$. When $T<\infty$, such a solution is called local in-time mild (weak) solution, otherwise, it is called global in-time mild (weak) solution. Obviously, each global in-time solution locally exists.
\begin{defn}[Mild solution] Let $T>0$ and $(u_0,u_1)\in H^1_0(\Omega)\times H^1(\Omega)$. A function $u$ is said to be a mild solution of \eqref{eq1} if 
	\begin{align*}
	u\in \ml{C}\left([0,T),H^1_0(\Omega)\right)\cap \ml{C}^{1}\left([0,T),H^1(\Omega)\right),
	\end{align*}
	and $u$ has initial data $u(0,x)=u_0(x)$, $u_t(0,x)=u_1(x)$ and satisfies the integral equation
	\begin{equation}\label{eq3.1}
	u(t,x)=R(t)(u_0,u_1)(x)+\int\nolimits_0^tS(t-s)f(u,u_t)(s,x)\,\mathrm{d}s
	\end{equation}
	in the sense of $H^1(\Omega)$.
\end{defn}
\begin{defn}[Weak solution] Let $T>0$ and $(u_0,u_1)\in L_{\mathrm{loc}}^1(\Omega)\times L_{\mathrm{loc}}^1(\Omega)$. A function $u$ is said to be a weak solution of \eqref{eq1} if
	\begin{align*}
	&u\in  W^{1,q}\left((0,T),L_{\mathrm{loc}}^q(\Omega)\right)\quad\qquad\qquad\qquad\qquad\quad\mbox{if}\,\,f(u,u_t)=|u_t|^q,\\
	&u\in L^p\left((0,T),L_{\mathrm{loc}}^p(\Omega)\right)\cap  W^{1,q}\left((0,T),L_{\mathrm{loc}}^q(\Omega)\right)\quad\mbox{if}\,\,f(u,u_t)=|u|^p+|u_t|^q,
	\end{align*}
	and $u$ has initial data $u(0,x)=u_0(x)$, $u_t(0,x)=u_1(x)$ and satisfies the relation
	\begin{align}
	&\int\nolimits_0^T\int\nolimits_{\Omega}f(u,u_t)(t,x)\varphi(t,x)\,\mathrm{d}x\,\mathrm{d}t+\int\nolimits_{\Omega}u_1(x)\varphi(0,x)\,\mathrm{d}x-\int\nolimits_{\Omega}u_0(x)\Delta\varphi(0,x)\,\mathrm{d}x-\int\nolimits_{\Omega}u_0(x)\varphi_t(0,x)\,\mathrm{d}x\notag\\
	&=\int\nolimits_0^T\int\nolimits_{\Omega}u(t,x)\varphi_{tt}(t,x)\,\mathrm{d}x\,\mathrm{d}t+\int\nolimits_0^T\int\nolimits_{\Omega}u(t,x)\Delta\varphi_t(t,x)\,\mathrm{d}x\,\mathrm{d}t-\int\nolimits_0^T\int\nolimits_{\Omega}u(t,x)\Delta\varphi(t,x)\,\mathrm{d}x\,\mathrm{d}t,\label{eq3.2}
	\end{align}
	for any compactly supported test function $\varphi\in \ml{C}^2([0,T]\times\Omega)$ such that $\varphi(T,x)=0$ and $\varphi_t(T,x)=0$.
\end{defn}
The following lemma is crucial for the proof of Theorems \ref{a=0} and  \ref{a>0}.
\begin{lem}[Mild $\rightarrow$ Weak]\label{mildweak} Let $(u_0,u_1)\in H_0^1(\Omega)\times H^1(\Omega)$. Under the assumption \eqref{assumption1}, if $u$ is a global (in time) mild solution of \eqref{eq1}, then $u$ is a global (in time) weak solution of \eqref{eq1}.
\end{lem}
\begin{proof} We now give the proof for the case when $f(u,u_t)=|u|^p+|u_t|^q$. For the remaining case $f(u,u_t)=|u_t|^q$, one may follow the next approach to directly obtain the desired result.
	
	Let $u$ be a global mild solution of \eqref{eq1}, $T_0>0$. Let $\varphi\in \ml{C}^2([0,T_0]\times\Omega)$ be a compactly supported function carrying the properties $\varphi(T_0,x)=0$ and $\varphi_t(T_0,x)=0$.
	
	It follows from the Gagliardo-Nirenberg inequality, under the assumption \eqref{assumption1}, that
	\begin{align}\label{eq3.6}
	\|f(u,u_t)(t,\cdot)\|_{L^2}&\leqslant  \|u(t,\cdot)\|^{p}_{L^{2p}} +\|u_t(t,\cdot)\|^{q}_{L^{2q}}\notag\\
	&\leqslant C\|\nabla u(t,\cdot)\|^{\sigma_1 p}_{L^2}\| u(t,\cdot)\|^{(1-\sigma_1) p}_{L^2}+ C\|\nabla u_t(t,\cdot)\|^{\sigma_2 q}_{L^2}\| u_t(t,\cdot)\|^{(1-\sigma_2) q}_{L^2}\notag\\
	&\leqslant C\|u\|^p_{\ml{C}([0,T_0],H_0^1(\Omega))}+C\|u\|^q_{\ml{C}^1([0,T_0],H^1(\Omega))},
	\end{align}
	where $\sigma_1=n(p-1)/(2p)\in(0,1]$ and $\sigma_2=n(q-1)/(2q)\in(0,1]$. This shows that $f(u,u_t)\in \ml{C}\left([0,T_0],L^2(\Omega)\right)$.
	
	Thanks to the density argument, there exist sequences 
	\begin{align*}
	\left\{\left(u_0^{(j)},u_1^{(j)}\right)\right\}_{j=1}^\infty\subseteq \left(H^2(\Omega)\cap H^1_0(\Omega)\right)^2,
	\end{align*}
	\begin{align*}
	\left\{F^{(j)}\right\}_{j=1}^\infty\subseteq \ml{C}\left([0,T_0],H^2(\Omega)\cap H^1_0(\Omega)\right)\cap \ml{C}^1\left([0,T_0],L^2(\Omega)\right)
	\end{align*}
	such that
	\begin{align*}
	\lim_{j\rightarrow\infty}\left(u_0^{(j)},u_1^{(j)}\right)=(u_0,u_1)\;\;\mbox{in}\;\; H_0^1(\Omega)\times H^1(\Omega)\quad\mbox{and}\quad \lim_{j\rightarrow\infty}F^{(j)}=f(u,u_t)\;\;\mbox{in}\;\; \ml{C}\left([0,T_0],L^2(\Omega)\right).
	\end{align*}
	Using Proposition \ref{prop2.2}, let $u^{(j)}$ be the strong solution of the linear inhomogeneous
	equation \eqref{eq2.4} with initial data $\left(u_0^{(j)},u_1^{(j)}\right)$
	and the inhomogeneous term $F^{(j)}(t,x)$. Using Proposition \ref{prop2.3} to $u^{(j)}$ and knowing the fact that $u$ is a mild solution of \eqref{eq1}, one may derive
	\begin{align*}
	u^{(j)}(t,x)-u(t,x)=R(t)\left(u^{(j)}_0-u_0,u^{(j)}_1-u_1\right)(x)+\int\nolimits_0^tS(t-s)\left(F^{(j)}(s,x)-f(u,u_t)(s,x)\right)\mathrm{d}s
	\end{align*}
	and hence, by using \eqref{eq2.3} in Proposition \ref{prop2.1}, it shows
	\begin{align*}
	&\left\|\left(u^{(j)}-u\right)(t,\cdot)\right\|_{L^2}\\
	&\leqslant \left\| R(t)\left(u^{(j)}_0-u_0,u^{(j)}_1-u_1\right)(\cdot)\right\|_{L^2}+\int\nolimits_0^t\left\|R(t-s)\left(0,F^{(j)}(s,\cdot)-f(u,u_t)(s,\cdot)\right)\right\|_{L^2}\,\mathrm{d}s\\
	&\leqslant C\left\|u^{(j)}_0-u_0\right\|_{L^2}+T_0\left\|\left(\nabla\left(u^{(j)}_0-u_0\right),u^{(j)}_1-u_1\right)\right\|_{L^2\times L^2}+T_0\int\nolimits_0^t\left\|F^{(j)}(s,\cdot)-f(u,u_t)(s,\cdot)\right\|_{L^2}\,\mathrm{d}s\\
	&\leqslant C(1+T_0)\left\|\left(u^{(j)}_0-u_0,u^{(j)}_1-u_1\right)\right\|_{H^1_0\times L^2}+T^2_0\sup_{s\in[0,T_0]}\left\|F^{(j)}(s,\cdot)-f(u,u_t)(s,\cdot)\right\|_{L^2},
	\end{align*}
	which implies, by letting $j\rightarrow\infty$, that 
	\begin{align*}
	u^{(j)}\longrightarrow u \;\;\mbox{in}\;\; \ml{C}\left([0,T_0],L^2(\Omega)\right).
	\end{align*}
	Moreover, due to the fact that $u^{(j)}$ is a strong solution of \eqref{eq2.4}, $u^{(j)}$ is also a weak solution of \eqref{eq2.4}, i.e., satisfies \eqref{eq2.6}.
	Thus, letting $j\rightarrow\infty$, we may deduce that $u$ satisfies the formulation \eqref{eq3.2}. Since $\varphi$ is an arbitrary test function, we claim that $u$ is a weak solution of \eqref{eq1}.
\end{proof}

Let us give the proof of Proposition \ref{prop1}.

\begin{proof}[Proof of Proposition \ref{prop1}] Similarly, we now give the proof for the case when $f(u,u_t)=|u|^p+|u_t|^q$. For the rest case $f(u,u_t)=|u_t|^q$, one may follow the next progress to immediately derive the corresponding desired result.
	
	Let $T>0$ and $R>0$. We now define the family of evolution space
	\begin{align*}
	Y_R(T):=\left\{v\in X(T):=\ml{C}\left([0,T],H_0^1(\Omega)\right)\cap \ml{C}^1\left([0,T],H^1(\Omega)\right)\quad \mbox{such that}\quad\|v\|_{X(T)}\leqslant 2R\right\},
	\end{align*}
	carrying
	\begin{align*}
	\|v\|_{X(T)}:=\sup_{t\in[0,T]}\left(\|v_t(t,\cdot)\|_{H^1}+\|v(t,\cdot)\|_{H^1}\right).
	\end{align*}
	As \eqref{eq3.6}, the application of the Gagliardo-Nirenberg inequality leads to
	\begin{align*}
	v\in Y_R(T)\rightarrow f(v,v_t)=|v|^p+|v_t|^q\in \ml{C}\left([0,T],L^2(\Omega)\right),
	\end{align*}
	which allow us by using Proposition \ref{prop2.4} defines a mapping 
	\begin{align*}
	\Phi:\, Y_R(T)\rightarrow X(T)
	\end{align*}
	such that $u(t,x)=\Phi(v)(t,x)$ is the unique mild solution to the linear inhomogeneous equation 
	\begin{equation*}
	\begin{cases}
	u_{tt}-\Delta
	u -\Delta u_t =f(v,v_t), &x\in {\Omega},\,t>0,\\
	u(0,x)=  u_0(x),\;u_t(0,x)=  u_1(x),&x\in {\Omega},\\
	u=0,&x\in {\partial\Omega},\,t>0.
	\end{cases}
	\end{equation*}
	Then, we also find
	\begin{align}
	&\|(u_t,\nabla u)(t,\cdot)\|_{L^2\times L^2}\leqslant C\|(u_1,\nabla
	u_0)\|_{L^2\times L^2}+C\int\nolimits_0^t\|f(v,v_t)(s,\cdot)\|_{L^2}\,\mathrm{d}s,\label{eq3.3}\\
	&\|u(t,\cdot)\|_{L^2}\leqslant \|u_0\|_{L^2}+C\int\nolimits_0^t\left(\|(u_1,\nabla
	u_0)\|_{L^2\times
		L^2}+\int\nolimits_0^s\|f(v,v_t)(\tau,\cdot)\|_{L^2}\,\mathrm{d}\tau\right)\mathrm{d}s,\label{eq3.4}\\
	&\|\nabla u_t(t,\cdot)\|^2_{L^2}\leqslant  C\|(\nabla
	u_0,u_1)\|^2_{L^2\times H^1}+C\int\nolimits_0^t\|f(v,v_t)(s,\cdot)\|_{L^2}^2\,\mathrm{d}s +C\|\nabla u(t,\cdot)\|^2_{L^2}\notag\\
	&\qquad\qquad\qquad\quad+C\int\nolimits_0^t \|f(v,v_t)(s,\cdot)\|_{L^2}\|u_s(s,\cdot)\|_{L^2}\,\mathrm{d}s.\label{eq3.5}
	\end{align}
	
	Next, we divide the proof into two steps to derive our result.
	
	\noindent$\bullet$  \underline{Step 1. Let us prove  $\Phi:\,Y_R(T)\rightarrow Y_R(T)$.}
	\medskip
	
	\noindent Let us consider $v\in Y_R(T)$ and $u=\Phi(v)$. Using \eqref{eq3.6}, one has
	\begin{align*}
	\|f(v,v_t)(s,\cdot)\|_{L^2}\leqslant C\,\|v\|^p_{X(T)}+C\,\|v\|^q_{X(T)}\leqslant C\,2^pR^p+ C\,2^qR^q,
	\end{align*}
	therefore, by using the inequalities \eqref{eq3.3}-\eqref{eq3.5}, we infer that
	\begin{align*}
	\|(u_t,\nabla u)(t,\cdot)\|_{L^2\times L^2}\leqslant C\,I_0+C\,2^p R^p T+C\,2^q R^q T,
	\end{align*}
	\begin{align*}
	\|u(t,\cdot)\|_{L^2}&\leqslant  \|u_0\|_{L^2}+C\int\nolimits_0^t\left(\|(u_1,\nabla u_0)\|_{L^2\times L^2}+C\,2^p R^p T+C\,2^q R^q T\right)\mathrm{d}s\\
	&\leqslant C(1+T)I_0+C\,2^p R^p T^2+C\,2^q R^q T^2\leqslant CI_0+C\,2^p R^p T+C\,2^q R^q T
	\end{align*}
	and
	\begin{align*}
	\|\nabla u_t(t,\cdot)\|^2_{L^2}&\leqslant C\,I^2_0+C\,2^{2p} R^{2p} T +C\,2^{2q} R^{2q} T +C\,I^2_0+C\,2^{2p} R^{2p} T^2\\
	&\quad+C\,2^{2q} R^{2q} T^2+C\left(2^{p} R^{p}+2^{q} R^{q}\right)\left(I_0+2^{p} R^{p} T+2^{q} R^{q} T\right)T\\
	&\leqslant C\,I^2_0+C\,2^{2p} R^{2p} T +C\,2^{2q} R^{2q} T +C\left(I_0+2^p R^p T+2^q R^q T\right)\left(2^p R^p+2^q R^q\right) T,
	\end{align*}
	where $I_0:=\|(u_0,u_1)\|_{H^1\times H^1}$, and $T\ll1$. Therefore, for any large constant $R$, we may choose sufficiently small constant $T$ such that $\|u\|_{X(T)}\leqslant 2R$. This proves that $\Phi$ is a mapping from $Y_R(T)$ to $Y_R(T)$.\\
	\medskip
	
	\noindent $\bullet$ \underline{Step 2. Let us prove $\Phi$ is a contraction.}
	\medskip
	
	\noindent Let $v,\overline{v}\in Y_R(T)$, $u:=\Phi(v)$ and $\overline{u}:=\Phi(\overline{v})$. Additionally, we define a new variable 
	\begin{align*}
	w:=u-\overline{u}.
	\end{align*} According to Proposition \ref{prop2.4}, the function $w$ is the unique mild solution to the linear inhomogeneous equation 
	\begin{equation*}
	\begin{cases}
	w_{tt}-\Delta
	w -\Delta w_t =f(v,v_t)-f(\overline{v},\overline{v}_t), &x\in {\Omega},\,t>0,\\
	w(0,x)= 0,\;w_t(0,x)= 0,&x\in {\Omega},\\
	w=0,&x\in {\partial\Omega},\,t>0,
	\end{cases}
	\end{equation*}
	and the following energy estimates hold:
	\begin{align*}
	&\|(w_t,\nabla w)(t,\cdot)\|_{L^2\times L^2}\leqslant C\int\nolimits_0^t\|f(v,v_t)(s,\cdot)-f(\overline{v},\overline{v}_t)(s,\cdot)\|_{L^2}\,\mathrm{d}s,\\
	&\|w(t,\cdot)\|_{L^2}\leqslant C\int\nolimits_0^t\int\nolimits_0^s\|f(v,v_t)(\tau,\cdot)-f(\overline{v},\overline{v}_t)(\tau,\cdot)\|_{L^2}\,\mathrm{d}\tau\,\mathrm{d}s,
	\end{align*}
	and
	\begin{align*}
	\|\nabla w_t(t,\cdot)\|^2_{L^2}&\leqslant  C\int\nolimits_0^t\|f(v,v_t)(s,\cdot)-f(\overline{v},\overline{v}_t)(s,\cdot)\|_{L^2}^2\,\mathrm{d}s +C\|\nabla w(t,\cdot)\|^2_{L^2}\\
	&\quad+C\int\nolimits_0^t \|f(v,v_t)(s,\cdot)-f(\overline{v},\overline{v}_t)(s,\cdot)\|_{L^2}\|w_s(s,\cdot)\|_{L^2}\,\mathrm{d}s.
	\end{align*}
	Using H\"{o}lder's inequality, Sobolev's embeddings $H_0^1(\Omega),H^1(\Omega)\hookrightarrow L^{2r}(\Omega)$ for $r>1$, and the following well-known inequality:
	\begin{align}\label{estimationimp}
	\left||x|^{r}-|y|^{r}\right|\leqslant C(r)|x-y|\left(|x|^{r-1}+|y|^{r-1}\right),\quad x,y\in\mathbb{R},\,r>1,
	\end{align}
	we conclude for all $t>0$ that
	\begin{align*}
	&\left\|f(v,v_t)(t,\cdot)-f(\overline{v},\overline{v}_t)(t,\cdot)\right\|_{L^2}\\
	&\leqslant \left\||v(t,\cdot)|^p-|\overline{v}(t,\cdot)|^p\right\|_{L^2}+ \left\||v_t(t,\cdot)|^q-|\overline{v}_t(t,\cdot)|^q\right\|_{L^2}\\
	&\leqslant C \left\||v(t,\cdot)-\overline{v}(t,\cdot)|\left(|v(t,\cdot)|^{p-1}+|\overline{v}(t,\cdot)|^{p-1}\right)\right\|_{L^2}\\
	&\quad+ \left\||v_t(t,\cdot)-\overline{v}_t(t,\cdot)|\left(|v_t(t,\cdot)|^{q-1}+|\overline{v}_t(t,\cdot)|^{q-1}\right)\right\|_{L^2}\\
	&\leqslant C \left\|v(t,\cdot)-\overline{v}(t,\cdot)\right\|_{L^{2p}}\left\||v(t,\cdot)|^{p-1}+|\overline{v}(t,\cdot)|^{p-1}\right\|_{L^{2p/(p-1)}}\\
	&\quad+ \left\|v_t(t,\cdot)-\overline{v}_t(t,\cdot)\right\|_{L^{2q}}\left\||v_t(t,\cdot)|^{q-1}+|\overline{v}_t(t,\cdot)|^{q-1}\right\|_{L^{2q/(q-1)}}\\
	&\leqslant C \left\|v(t,\cdot)-\overline{v}(t,\cdot)\right\|_{L^{2p}}\left(\|v(t,\cdot)\|_{L^{2p}}^{p-1}+\|\overline{v}(t,\cdot)\|_{L^{2p}}^{p-1}\right)\\
	&\quad+ \left\|v_t(t,\cdot)-\overline{v}_t(t,\cdot)\right\|_{L^{2q}}\left(\|v_t(t,\cdot)\|_{L^{2q}}^{q-1}+\|\overline{v}_t(t,\cdot)\|_{L^{2q}}^{q-1}\right)\\
	&\leqslant C \left\|v(t,\cdot)-\overline{v}(t,\cdot)\right\|_{H_0^1}\left(\|v(t,\cdot)\|_{H_0^1}^{p-1}+\|\overline{v}(t,\cdot)\|_{H_0^1}^{p-1}\right)\\
	&\quad+ \left\|v_t(t,\cdot)-\overline{v}_t(t,\cdot)\right\|_{H^1}\left(\|v_t(t,\cdot)\|_{H^1}^{q-1}+\|\overline{v}_t(t,\cdot)\|_{H^1}^{q-1}\right)\\
	&\leqslant C \left\|v-\overline{v}\right\|_{X(T)}\left(\|v\|_{X(T)}^{p-1}+\|\overline{v}\|_{X(T)}^{p-1}+\|v\|_{X(T)}^{q-1}+\|\overline{v}\|_{X(T)}^{q-1}\right)\\
	&\leqslant C \left(2^pR^{p-1}+2^qR^{q-1}\right) \|v-\overline{v}\|_{X(T)}.
	\end{align*}
	Therefore, similarly as the above and by choosing sufficiently small constant $T$ for any large constant $R$, we may conclude that 
	\begin{align*}
	\|w\|_{X(T)}\leqslant  \frac{1}{2}\|v-\overline{v}\|_{X(T)}.
	\end{align*}
	This implies that $\Phi$ is a contraction mapping. Then, according to the Banach fixed-point
	theorem, there exists a unique mild solution $u\in X(T)$ to problem \eqref{eq1}.
	
	Moreover, by uniqueness, there exists a maximal interval $\left[0,T_{\max}\right)$, where
	\begin{align*}
	T_{\max}:=\sup\left\{T>0:\;\mbox{there exist a mild solution $u\in X(T)$
		to \eqref{eq1}}\right\}\leqslant +\infty.
	\end{align*}
	
	Finally, if the lifespan $T_{\max}$ is finite, then the energy of the solution blows up at $T_{\max}$ such that
	\begin{align*}
	\lim_{t\rightarrow T_{\max}}\left(\|u(t,\cdot)\|_{H_0^1}+\|u_t(t,\cdot)\|_{H^1}\right)=\infty.
	\end{align*}
	Indeed, providing that
	\begin{align*}
	\lim_{t\rightarrow T_{\max}}\left(\|u(t,\cdot)\|_{H_0^1}+\|u_t(t,\cdot)\|_{H^1}\right)=:M<\infty,
	\end{align*}
	then there exists a time sequence $\{t_m\}_{m\geqslant 0}$ tending to $T_{\max}$ as $m\rightarrow\infty$ and such that
	\begin{align*}
	\sup_{m\in\mathbb{N}}\left(\|u(t_m,\cdot)\|_{H_0^1}+\|u_t(t_m,\cdot)\|_{H^1}\right)\leqslant M+1.
	\end{align*}
	The argument before shows that there exists $T(M + 1) > 0$ such that the solution
	$u(\cdot,x)$ can be extended on the interval $[t_m, t_m + T(M + 1)]$ for any $m$. By taking
	$m$ sufficiently large so that $t_m \geqslant T_{\max} - (1/2)T(M + 1)$, the solution $u(\cdot,x)$ can be
	extended on $[T_{\max}, T_{\max} + (1/2)T(M + 1)]$. This contradicts the definition of $T_{\max}$. Thus, we complete the proof.
\end{proof}


\section{Blow-up of solutions}\label{sec4}
This section is devoted to prove the blow-up results for \eqref{eq1}, namely, Theorems \ref{a=0} and \ref{a>0}. The main approach of the proof is based on the variational formulation of the weak solution by choosing the appropriate test functions. Note that the harmonic functions introduced in Lemmas \ref{lemma1}, \ref{lemma4} and \ref{lemma2} play a crucial role in an exterior domain, because of their asymptotic behaviors and the value vanishing on the boundary $\partial\Omega$.

\begin{proof}[Proof of Theorem \ref{a=0}]
	We argue by a contradiction that assuming that $u$ is 
	a global (in time) solution of \eqref{eq1}. It immediately shows the following relation:
	\begin{align}\label{newweaksolution}
	&\int\nolimits_0^T\int\nolimits_{\Omega}|u_t(t,x)|^{q}\varphi(t,x)\,\mathrm{d}x\,\mathrm{d}t+\int\nolimits_{\Omega}u_1(x)\varphi(0,x)\,\mathrm{d}x- \int\nolimits_{\Omega} u_0(x)\Delta\varphi(0,x)\,\mathrm{d}x-\int\nolimits_{\Omega}u_0(x)\varphi_t(0,x)\,\mathrm{d}x\notag\\
	&=\int\nolimits_0^T\int\nolimits_{\Omega}u(t,x)\varphi_{tt}(t,x)\,\mathrm{d}x\,\mathrm{d}t+\int\nolimits_0^T\int\nolimits_{\Omega}u(t,x)\Delta\varphi_t(t,x)\,\mathrm{d}x\,\mathrm{d}t-\int\nolimits_0^T\int\nolimits_{\Omega}u(t,x)\Delta\varphi(t,x)\,\mathrm{d}x\,\mathrm{d}t
	\end{align}
	for all $T>0$ and all compactly supported function $\varphi\in \ml{C}^2([0,T]\times\Omega)$ 
	such that $\varphi(T,x)=0$ and $\varphi_t(T,x)=0$ for all $x\in\Omega$. \\
	Let us take a test function
	\begin{align*}
	\varphi(t,x):=\phi_0(x)\varphi^\ell_T(x)\eta_T^{k}(t)
	\end{align*} with $\ell, k\gg1$, where $\phi_0$ is the harmonic function introduced in Lemmas \ref{lemma1}, \ref{lemma4} and \ref{lemma2}. (It depends on different dimensions) On one hand, $\eta_T(t):=\eta(t/T)$, where $\eta\in \ml{C}^\infty([0,\infty])$ is a non-increasing cut-off function such that
	\begin{equation*}
	\eta(t):=\begin{cases}
	1&\mbox{if }0\leqslant t\leqslant 1/2\\
	0&\mbox {if } t\geqslant 1,
	\end{cases}
	\end{equation*}
	carrying $0\leqslant \eta(t) \leqslant 1$ and $|\eta'(t)|\leqslant C$ for some constants $C>0$ and all $t>0$. On the other hand, $\varphi_T(x)=\Phi(|x|/T)$ with the following smooth, non-increasing cut-off function:
	\begin{equation*}
	\Phi(r):=\begin{cases}
	1&\mbox{if }0\leqslant r\leqslant 1,\\
	0&\mbox {if } r\geqslant 2,
	\end{cases}
	\end{equation*}
	such that $0\leqslant\Phi(r)\leqslant 1$, $|\Phi'(r)|\leqslant C/r$ and $|\Phi''(r)|\leqslant C/r^2$ with some constants $C>0$. Finally, let us define an additional test function $\Psi_T=\Psi_T(t)$ such that
	\begin{align*}
	\Psi_T(t):=\int\nolimits_t^{\infty}\eta^k_T(\tau)\,\mathrm{d}\tau.
	\end{align*}
	This test function has properties $\Psi_T'(t)=-\eta^k_T(t)$, and $\hbox{supp }\Psi_T\subseteq[0,T]$.
	
	Making use of the properties of these test function, we may derive
	\begin{align}\label{weaksolution2}
	&\int\nolimits_0^T\int\nolimits_{\Omega_1}|u_t(t,x)|^{q}\varphi(t,x)\,\mathrm{d}x\,\mathrm{d}t+\int\nolimits_{\Omega_1}u_0(x)\Psi_T(0)\Delta\left(\phi_0(x)\varphi_T^\ell(x)\right)\mathrm{d}x+\int\nolimits_{\Omega_1}u_1(x)\phi_0(x)\varphi_T^\ell(x)\,\mathrm{d}x\notag\\
	&=-\int\nolimits_0^T\int\nolimits_{\Omega_1}u_t(t,x)\phi_0(x)\varphi_T^\ell(x)\partial_t\left(\eta^k_T(t)\right)\mathrm{d}x\,\mathrm{d}t-\int\nolimits_0^T\int\nolimits_{\Omega_1}u_t(t,x)\Delta\left(\phi_0(x)\varphi_T^\ell(x)\right)\eta^k_T(t)\,\mathrm{d}x\,\mathrm{d}t\notag\\
	&\quad-\int\nolimits_0^T\int\nolimits_{\Omega_1}u_t(t,x)\Delta\left(\phi_0(x)\varphi_T^\ell(x)\right)\Psi_T(t)\,\mathrm{d}x\,\mathrm{d}t\notag\\
	&=:I_1+I_2+I_3
	\end{align}
	where $\Omega_1:=\{x\in\Omega:\;|x|\leqslant 2T\}$. At this stage, we have to distinguishes three cases such that $n\geqslant 3$, $n=2$ and $n=1$. In each case, we will apply different asymptotic properties of the harmonic function $\phi_0(x)$.
	
	\medskip
	
	\noindent $\bullet$ \underline{Proof of blow-up for $n\geqslant 3$.}
	\medskip
	
	\noindent In order to estimate the right-hand side of \eqref{weaksolution2}, we introduce the term 
	$\varphi^{1/q}\varphi^{-1/q}$ in $I_1$, and we use Young's inequality to obtain
	\begin{align}\label{I1}
	I_1
	\leqslant\frac{1}{6} \int\nolimits_0^T\int\nolimits_{\Omega_1}|u_t(t,x)|^q\varphi(t,x) \,\mathrm{d}x\,\mathrm{d}t+C\int\nolimits_0^T\int\nolimits_{\Omega_1}\phi_0(x)\varphi_T^\ell(x)\eta_T^{(k-1)q'}(t)\left|\partial_t \eta_T(t)\right|^{q'}\mathrm{d}x\,\mathrm{d}t.
	\end{align}
	Let us consider Lemma \ref{lemma1} with all properties of $\phi_0$, $T\gg1$, and Young's inequality, which deduce
	\begin{align}\label{I2}
	I_2
	&\leqslant\frac{1}{6} \int\nolimits_0^T\int\nolimits_{\Omega_1}|u_t(t,x)|^q\varphi (t,x)\,\mathrm{d}x\,\mathrm{d}t+C\int\nolimits_0^T\int\nolimits_{\nabla\Omega_1}\varphi_T^{\ell-q'}(x)\eta_T^{k}(t)|\nabla\phi_0(x)|^{q'}\left|\nabla\varphi_T(x)\right|^{q'}\mathrm{d}x\,\mathrm{d}t\notag\\
	&\quad+\,C\int\nolimits_0^T\int\nolimits_{\nabla\Omega_1}\varphi_T^{\ell-2q'}(x)\eta_T^{k}(t)
	\left|\nabla\varphi_T(x)\right|^{2q'}\mathrm{d}x\,\mathrm{d}t\notag\\
	&\quad+\,C\int\nolimits_0^T\int\nolimits_{\nabla\Omega_1}\varphi_T^{\ell-q'}(x)\eta_T^{k}(t)
	\left|\Delta\varphi_T(x)\right|^{q'}\mathrm{d}x\,\mathrm{d}t,
	\end{align}
	where $\nabla\Omega_1:=\{x\in\Omega:\;T\leqslant |x|\leqslant 2T\}$. Similarly,
	\begin{align}\label{I3}
	I_3&\leqslant\frac{1}{6} \int\nolimits_0^T\int\nolimits_{\Omega_1}|u_t(t,x)|^q\varphi(t,x) \,\mathrm{d}x\,\mathrm{d}t\notag\\
	&\quad+C\int\nolimits_0^T\int\nolimits_{\nabla\Omega_1}\varphi_T^{\ell-q'}(x)\eta^{-kq'/q}_T(t)\Psi^{q'}_T(t)|\nabla\phi_0(x)|^{q'}\left|\nabla\varphi_T(x)\right|^{q'}\mathrm{d}x\,\mathrm{d}t\notag\\
	&\quad+\,C\int\nolimits_0^T\int\nolimits_{\nabla\Omega_1}\varphi_T^{\ell-2q'}(x)\eta^{-kq'/q}_T(t)\Psi^{q'}_T(t)
	\left|\nabla\varphi_T(x)\right|^{2q'}\mathrm{d}x\,\mathrm{d}t\notag\\
	&\quad+\,C\int\nolimits_0^T\int\nolimits_{\nabla\Omega_1}\varphi_T^{\ell-q'}(x)\eta^{-kq'/q}_T(t)\Psi^{q'}_T(t)
	\left|\Delta\varphi_T(x)\right|^{q'}\mathrm{d}x\,\mathrm{d}t.
	\end{align}
	Using \eqref{I1}-\eqref{I3}, it follows from \eqref{weaksolution2} that
	\begin{align}\label{weaksolution3}
	&\frac{1}{2}\int\nolimits_0^T\int\nolimits_{\Omega_1}|u_t(t,x)|^{q}\varphi(t,x)\,\mathrm{d}x\,\mathrm{d}t+\int\nolimits_{\Omega_1}u_0(x)\Psi_T(0)\Delta\left(\phi_0(x)\varphi_T^\ell(x)\right)\mathrm{d}x+\int\nolimits_{\Omega_1}u_1(x)\phi_0(x)\varphi_T^\ell(x)\,\mathrm{d}x\notag\\
	& \leqslant C\int\nolimits_0^T\int\nolimits_{\Omega_1}\phi_0(x)\varphi_T^\ell(x)\eta_T^{(k-1)q'}(t)\left|\partial_t \eta_T(t)\right|^{q'}\mathrm{d}x\,\mathrm{d}t\notag\\
	&\quad+\,C\int\nolimits_0^T\int\nolimits_{\nabla\Omega_1}\varphi_T^{\ell-q'}(x)\eta_T^{k}(t)|\nabla\phi_0(x)|^{q'}\left|\nabla\varphi_T(x)\right|^{q'}\mathrm{d}x\,\mathrm{d}t\notag\\
	&\quad+\;C\int\nolimits_0^T\int\nolimits_{\nabla\Omega_1}\varphi_T^{\ell-2q'}(x)\eta_T^{k}(t)
	\left|\nabla\varphi_T(x)\right|^{2q'}\mathrm{d}x\,\mathrm{d}t+\,C\int\nolimits_0^T\int\nolimits_{\nabla\Omega_1}\varphi_T^{\ell-q'}(x)\eta_T^{k}(t)
	\left|\Delta\varphi_T(x)\right|^{q'}\mathrm{d}x\,\mathrm{d}t\notag\\
	&\quad+\,C\int\nolimits_0^T\int\nolimits_{\nabla\Omega_1}\varphi_T^{\ell-q'}(x)\eta^{-kq'/q}_T(t)\Psi^{q'}_T(t)|\nabla\phi_0(x)|^{q'}\left|\nabla\varphi_T(x)\right|^{q'}\mathrm{d}x\,\mathrm{d}t\notag\\
	&\quad+\,C\int\nolimits_0^T\int\nolimits_{\nabla\Omega_1}\varphi_T^{\ell-2q'}(x)\eta^{-kq'/q}_T(t)\Psi^{q'}_T(t)
	\left|\nabla\varphi_T(x)\right|^{2q'}\mathrm{d}x\,\mathrm{d}t\notag\\
	&\quad+\,C\int\nolimits_0^T\int\nolimits_{\nabla\Omega_1}\varphi_T^{\ell-q'}(x)\eta^{-kq'/q}_T(t)\Psi^{q'}_T(t)
	\left|\Delta\varphi_T(x)\right|^{q'}\mathrm{d}x\,\mathrm{d}t,
	\end{align}
	We now divide the discussion into two cases.
	
	For the case when $1<q<1+1/n$, by Lemma \ref{lemma1} we know
	\begin{align*}
	|\nabla\phi_0(x)|\leqslant C|x|^{1-n}\leqslant CT^{1-n}\leqslant  C T^{-1}
	\end{align*}
	for $x\in\nabla\Omega_1$. Therefore, using the fact that 
	\begin{align*}
	\eta^{-kq'/q}_T(t)\Psi^{q'}_T(t)\leqslant \eta^{-kq'/q}_T(t)\left(\int\nolimits_t^T\eta^{k}_T(\tau)\,\mathrm{d}\tau\right)^{q'}\leqslant 
	T^{q'} \eta^{k}_T(t)\leqslant T^{q'}
	\end{align*}
	and the change of variables such that $y=T^{-1}x,$ $s=T^{-1}t$, we get from \eqref{weaksolution3} that 
	\begin{align}\label{weaksolution4}
	&\int\nolimits_{\Omega_1}u_0(x)\Psi_T(0)\Delta\left(\phi_0(x)\varphi_T^\ell(x)\right)\mathrm{d}x+\int\nolimits_{\Omega_1}u_1(x)\phi_0(x)\varphi_T^\ell(x)\,\mathrm{d}x\nonumber\\
	& \leqslant C\;T^{-q'+1+n}\int\nolimits_0^1\int\nolimits_{|y|\leqslant 2}\Phi^\ell(|y|)\eta^{(k-1)q'}(s)\left| \eta'(s)\right|^{q'}\mathrm{d}y\,\mathrm{d}s\nonumber\\
	&\quad +\,C\;T^{-2q'+1+n}\int\nolimits_0^1\int\nolimits_{1\leqslant |y|\leqslant 2}\Phi^{\ell-q'}(|y|)\eta^{k}(s)\left(\left|\nabla_y\Phi(|y|)\right|^{q'}+\left|\Delta_y\Phi(|y|)\right|^{q'}\right)\mathrm{d}y\,\mathrm{d}s\nonumber\\
	&\quad+\,C\;T^{-2q'+1+n}\int\nolimits_0^1\int\nolimits_{1\leqslant |y|\leqslant 2}\Phi^{\ell-2q'}(|y|)\eta^{k}(s)\left|\nabla_y\Phi(|y|)\right|^{2q'}\mathrm{d}y\,\mathrm{d}s\nonumber\\
	&\quad+\,C\;T^{-q'+1+n}\int\nolimits_0^1\int\nolimits_{1\leqslant |y|\leqslant 2}\Phi^{\ell-q'}(|y|)\left(\left|\nabla_y\Phi(|y|)\right|^{q'}+\left|\Delta_y\Phi(|y|)\right|^{q'}\right)\mathrm{d}y\,\mathrm{d}s\nonumber\\
	&\quad+\,C\;T^{-q'+1+n}\int\nolimits_0^1\int\nolimits_{1\leqslant |y|\leqslant 2}\Phi^{\ell-2q'}(|y|)
	\left|\nabla_y\Phi(|y|)\right|^{2q'}\mathrm{d}y\,\mathrm{d}s\nonumber\\
	&\leqslant \,C\;T^{-q'+1+n},
	\end{align}
	where the constant $C$ is independent of $T$. Since $q<1+1/n$, it follows by letting $T\rightarrow\infty$ that
	\begin{align*}
	\lim_{T\rightarrow\infty}\int\nolimits_{\nabla\Omega_1}u_0(x)\Psi_T(0)\Delta\left(\phi_0(x)\varphi_T^\ell(x)\right)\mathrm{d}x+\lim_{T\rightarrow\infty}\int\nolimits_{\Omega_1}u_1(x)\phi_0(x)\varphi_T^\ell(x)\,\mathrm{d}x\leqslant 0.
	\end{align*}
	On the other hand, since $\Psi_T(0)\leqslant C\,T$, $|\nabla\varphi_T^{\ell}(x)|\leqslant C/T$, $|\Delta\varphi_T(x)|\leqslant C/T^2$, and $|\nabla\phi_0(x)|\leqslant C/T$ in $\nabla\Omega_1$, we may observe that
	\begin{align*}
	\left|\Psi_T(0)\Delta\left(\phi_0(x)\varphi_T^\ell(x)\right)\right|\leqslant\frac{C}{T}
	\end{align*}
	for $x\in\nabla\Omega_1$. Moreover, according to $(u_0,u_1\phi_0)\in L^1(\Omega)\times L^1(\Omega)$, it follows by Lebesgue's dominated convergence theorem that
	\begin{align*}
	0<\int\nolimits_{\Omega}u_1(x)\phi_0(x)\,\mathrm{d}x\leqslant0,
	\end{align*}
	which leads to a contradiction.
	\medskip
	
	Let us consider the case $q=1+1/n$. From \eqref{weaksolution3}, there exists a positive constant $D$ independent of $T$ such that
	\begin{align*}
	\int\nolimits_0^{T}\int\nolimits_{\Omega_1}|u_t(t,x)|^q\varphi(t,x) \,\mathrm{d}x\,\mathrm{d}t\leqslant D\quad\mbox{for all}\;\; T>0,
	\end{align*}
	which implies that
	\begin{align}\label{lefthand}
	\int\nolimits_{T/2}^{T}\int\nolimits_{\Omega_1}|u_t(t,x)|^q\varphi(t,x) \,\mathrm{d}x\,\mathrm{d}t\longrightarrow 0\quad\mbox{and}\quad\int\nolimits_{0}^{T}\int\nolimits_{\nabla\Omega_1}|u_t(t,x)|^q\varphi(t,x) \,\mathrm{d}x\,\mathrm{d}t\longrightarrow 0,
	\end{align}
	as $T\rightarrow\infty$. On the other hand, we use H\"older's inequality instead of Young's one in $I_1$, $I_2$, and $I_3$, together with the same change of variables, we get
	\begin{align}\label{newI1}
	I_1&\leqslant\left(\int\nolimits_{T/2}^T\int\nolimits_{\Omega_1}|u_t(t,x)|^q\varphi(t,x) \,\mathrm{d}x\,\mathrm{d}t\right)^{1/p}\left(C\int\nolimits_0^T\int\nolimits_{\Omega_1}\phi_0(x)\varphi_T^\ell(x)\eta_T^{(k-1)q'}(t)\left|\partial_t \eta_T(t)\right|^{q'}\mathrm{d}x\,\mathrm{d}t\right)^{1/q'}\notag\\
	&\leqslant C\;T^{-1+\frac{1+n}{q'}}\left(\int\nolimits_{T/2}^T\int\nolimits_{\Omega_1}|u_t(t,x)|^q\varphi (t,x) \,\mathrm{d}x\,\mathrm{d}t\right)^{1/q}\notag\\
	&=C\;\left(\int\nolimits_{T/2}^T\int\nolimits_{\Omega_1}|u_t(t,x)|^q\varphi(t,x) \,\mathrm{d}x\,\mathrm{d}t\right)^{1/q},
	\end{align}
	thanks to $q=1+1/n$. Similarly, one gets 
	\begin{align}\label{newI2}
	I_2,\,I_3\leqslant C\;\left(\int\nolimits_{0}^T\int\nolimits_{\nabla\Omega_1}|u_t(t,x)|^q\;\varphi(t,x) \,\mathrm{d}x\,\mathrm{d}t\right)^{1/q}.
	\end{align}
	Finally, using \eqref{newI1} and \eqref{newI2}, it leads from \eqref{weaksolution2} that
	\begin{align*}
	&\int\nolimits_{\Omega_1}u_0(x)\Psi_T(0)\Delta\left(\phi_0(x)\varphi_T^\ell(x)\right)\mathrm{d}x+\int\nolimits_{\Omega_1}u_1(x)\phi_0(x)\varphi_T^\ell(x)\,\mathrm{d}x\\
	&\leqslant C\;\left(\int\nolimits_{T/2}^T\int\nolimits_{\Omega_1}|u_t(t,x)|^q\varphi(t,x) \,\mathrm{d}x\,\mathrm{d}t\right)^{1/q}+C\;\left(\int\nolimits_{0}^T\int\nolimits_{\nabla\Omega_1}|u_t(t,x)|^q\varphi(t,x) \,\mathrm{d}x\,\mathrm{d}t\right)^{1/q}.
	\end{align*}
	Hence, by letting $T\rightarrow\infty$ and using \eqref{lefthand}, we get a contradiction.
	\medskip
	
	\noindent $\bullet$ \underline{Proof of blow-up for $n=2$.}
	\medskip
	
	\noindent In this case, we have a blow-up result just in the case $1<q<1+1/n=3/2$. By repeating the same calculations in the case of $n\geqslant3$ and using Lemma \ref{lemma4} instead of Lemma \ref{lemma1} (note that the main difference is that $\phi_0(x)\leqslant C\ln (|x|)$), we easily conclude that
	\begin{align*}
	I_1&\leqslant \frac{1}{6} \int\nolimits_0^T\int\nolimits_{\Omega_1}|u_t(t,x)|^q\varphi(t,x) \,\mathrm{d}x\,\mathrm{d}t+ C\ln(T)\,T^{-q'+3},\\
	I_2&\leqslant \frac{1}{6} \int\nolimits_0^T\int\nolimits_{\Omega_1}|u_t(t,x)|^q\varphi(t,x) \,\mathrm{d}x\,\mathrm{d}t+ C\,T^{-2q'+3}+C\,\ln(T)\,T^{-2q'+3},\\
	I_3&\leqslant \frac{1}{6} \int\nolimits_0^T\int\nolimits_{\Omega_1}|u_t(t,x)|^q\varphi(t,x) \,\mathrm{d}x\,\mathrm{d}t+ C\,T^{-q'+3}+C\,\ln(T)\,T^{-q'+3}.
	\end{align*}
	This implies that
	\begin{align*}
	\int\nolimits_{\Omega_1}u_0(x)\Psi_T(0)\Delta\left(\phi_0(x)\varphi_T^\ell(x)\right)\mathrm{d}x+\int\nolimits_{\Omega_1}u_1(x)\phi_0(x)\varphi_T^\ell(x)\,\mathrm{d}x\leqslant C\ln(T)\,T^{-q'+3}\leqslant C\,T^{(-q'+3)/2},
	\end{align*}
	where we have used, e.g., the fact that $\ln(T)\leqslant C\,T^{(q'-3)/2}$. By letting $T$ goes to infinity and using our assumption $q<3/2$, we obtain the desired contradiction.\medskip
	
	\noindent $\bullet$ \underline{Proof of blow-up for $n=1$.}
	\medskip
	
	\noindent For the case $1<q<\frac{2+\sqrt{5}}{1+\sqrt{5}}=\frac{2\alpha+1}{2\alpha}$, where $\alpha:=\frac{1+\sqrt{5}}{2}$ is the positive root of the quadratic equation $\alpha^2-\alpha-1=0$, following the similar procedure as in the case of $n\geqslant 3$ and making use of Lemma \ref{lemma2} rather than Lemma \ref{lemma1}, the following estimates hold:
	\begin{align*}
	I_1&\leqslant \frac{1}{6} \int\nolimits_0^T\int\nolimits_{\Omega_1}|u_t(t,x)|^q\varphi(t,x) \,\mathrm{d}x\,\mathrm{d}t+ C\,T^{-q'+2\alpha+1},\\
	I_2&\leqslant \frac{1}{6} \int\nolimits_0^T\int\nolimits_{\Omega_1}|u_t(t,x)|^q\varphi(t,x) \,\mathrm{d}x\,\mathrm{d}t+C\,T^{-\alpha q'+\alpha+1}+C\,T^{-2\alpha q'+2\alpha+1},\\
	I_3&\leqslant \frac{1}{6} \int\nolimits_0^T\int\nolimits_{\Omega_1}|u_t(t,x)|^q\varphi(t,x) \,\mathrm{d}x\,\mathrm{d}t +C\,T^{-(\alpha-1)q'+\alpha+1}+C\,T^{-(2\alpha-1)q'+2\alpha+1}.
	\end{align*}
	Using the change of variables: $y=T^{-\alpha}x,$ $s=T^{-1}t$, we get a contradiction from $\eqref{weaksolution3}$ by letting $T\rightarrow\infty$.\\
	For the critical case $q=\frac{2+\sqrt{5}}{1+\sqrt{5}}$, we get the contradiction by applying a similar calculation as in the case $n\geqslant 3$ above by taking into account the support of $\partial_x\varphi_T(x)$, $\partial_x^2\varphi_T(x)$ and $\partial_t\eta_T(t)$. This completes the proof of Theorem $\ref{a=0}$.
\end{proof}
\medskip

\begin{proof}[Proof of Theorem \ref{a>0}]
	We may claim by contradiction that $u$ is 
	a global (in time) solution of \eqref{eq1}. Let us apply some integration by parts and some properties of test function to deduce the next equality:
	\begin{align}\label{newweaksolution(a>0)}
	&I+J+\int\nolimits_{\Omega}u_1(x)\varphi(0,x)\,\mathrm{d}x-\int\nolimits_{\Omega}u_0(x)\Delta \varphi(0,x)\,\mathrm{d}x-\int\nolimits_{\Omega}u_0(x)\varphi_t(0,x)\,\mathrm{d}x\notag\\
	&=\int\nolimits_0^T\int\nolimits_{\Omega}u(t,x)\varphi_{tt}(t,x)\,\mathrm{d}x\,\mathrm{d}t+\int\nolimits_0^T\int\nolimits_{\Omega}u(t,x)\Delta\varphi_t(t,x)\,\mathrm{d}x\,\mathrm{d}t-\int\nolimits_0^T\int\nolimits_{\Omega}u(t,x)\Delta\varphi(t,x)\,\mathrm{d}x\,\mathrm{d}t\notag\\
	&=:I_1+I_2+I_3,
	\end{align}
	where we denote
	\begin{align*}I:=\int\nolimits_0^T\int\nolimits_{\Omega}|u(t,x)|^{p}\varphi(t,x)\,\mathrm{d}x\,\mathrm{d}t\quad\mbox{and}\quad J:=\int\nolimits_0^T\int\nolimits_{\Omega}|u_t(t,x)|^{q}\varphi(t,x)\,\mathrm{d}x\,\mathrm{d}t,
	\end{align*}
	for all $T>0$ and all compactly supported function $\varphi\in \ml{C}^2([0,T]\times\Omega)$ 
	such that $\varphi(T,x)=0$ and $\varphi_t(T,x)=0$. We define the test function
	\begin{align*}
	\varphi(t,x):=\phi_0(x)\varphi^\ell_T(x)\eta_T^{k}(t),
	\end{align*}
	where $\phi_0$ is defined in Lemma \ref{lemma1}, \ref{lemma4} and \ref{lemma2}, the test functions $\varphi_T(t),\eta_T(x)$ are the same in the proof of Theorem \ref{a=0}.
	
	Again, we discuss the proof into three parts: $n\geqslant 3$, $n=2$ and $n=1$, respectively.\medskip
	
	\noindent $\bullet$ \underline{Proof of blow-up for $n\geqslant 3$.}
	\medskip
	
	Now, we should find a suitable combination of integration by parts. In other words, we will answer how to do integration by parts on the right-hand sides of the equation in \eqref{newweaksolution(a>0)}. Thus, all the possibilities should be shown.
	
	Applying Lemma \ref{lemma1} and Young's inequality, we can derive the following estimates for $I_j$ with $j=1,2,3$:\\
	\noindent \underline{Part I:} The possibilities of estimates of $I_1$ are
	\begin{align*}
	\mathbb{P}_1^{(1)}:\quad &I_1=\int\nolimits_0^{T}\int\nolimits_{\Omega}u(t,x)\phi_0(x)\partial_t^2\left(\eta^k_T(t)\right)\varphi^\ell_T(x)\,\mathrm{d}x\,\mathrm{d}t\leqslant\varepsilon I+C\,T^{n+1-2p'},\\
	\mathbb{P}_1^{(2)}:\quad &I_1=-\int\nolimits_0^{T}\int\nolimits_{\Omega}u_t(t,x)\phi_0(x)\partial_t\left(\eta^k_T(t)\right)\varphi^\ell_T(x)\,\mathrm{d}x\,\mathrm{d}t\leqslant\varepsilon J+C\,T^{n+1-q'}.
	\end{align*}
	\noindent \underline{Part II:} The possibilities of the estimates of $I_2$.
	\begin{align*}
	\mathbb{P}_2^{(1)}:\quad &I_2=\int\nolimits_0^{T}\int\nolimits_{\Omega}u(t,x)\partial_t\left(\eta^k_T(t)\right)\Delta\left(\phi_0(x)\varphi^\ell_T(x)\right)\mathrm{d}x\,\mathrm{d}t  \leqslant\varepsilon I+C\,T^{n+1-3p'},\\
	\mathbb{P}_2^{(2)}:\quad  &I_2=-\int\nolimits_0^{T}\int\nolimits_{\Omega}u_t(t,x)\eta^k_T(t)\Delta\left(\phi_0(x)\varphi^\ell_T(x)\right)\mathrm{d}x\,\mathrm{d}t-\int\nolimits_{\Omega}u_0(x)\Delta\left(\phi_0(x)\varphi^\ell_T(x)\right)\mathrm{d}x\\
	&\quad\leqslant \varepsilon J+C\,T^{n+1-2q'}-\int\nolimits_{\Omega}u_0(x)\Delta\left(\phi_0(x)\varphi^\ell_T(x)\right)\mathrm{d}x.
	\end{align*}
	\noindent \underline{Part III:} The possibilities of the estimates of $I_3$.
	\begin{align*}
	\mathbb{P}_3^{(1)}:\quad &I_3=-\int\nolimits_0^{T}\int\nolimits_{\Omega}u(t,x)\eta^k_T(t)\Delta\left(\phi_0(x)\varphi^\ell_T(x)\right)\mathrm{d}x\,\mathrm{d}t  \leqslant\varepsilon I+C\,T^{n+1-2p'},\\
	\mathbb{P}_3^{(2)}:\quad  &I_3=-\int\nolimits_0^{T}\int\nolimits_{\Omega}u_t(t,x)\Psi_T(t)\Delta\left(\phi_0(x)\varphi^\ell_T(x)\right)\mathrm{d}x\,\mathrm{d}t-\int\nolimits_{\Omega}u_0(x)\Psi_T(0)\Delta\left(\phi_0(x)\varphi^\ell_T(x)\right)\mathrm{d}x\\
	&\quad\leqslant\varepsilon J+C\,T^{n+1-q'}-\int\nolimits_{\Omega}u_0(x)\Psi_T(0)\Delta\left(\phi_0(x)\varphi^\ell_T(x)\right)\mathrm{d}x.
	\end{align*}
	In the above, we take the positive constant $\varepsilon\in(0,1/6)$.
	
	Now, we need to distinguish between the following eight cases.
	\begin{table}[http]
		\centering
		\begin{tabular}{|c | c | c| c |}
			\hline
			& Estimate for $I_1$ & Estimate for $I_2$  & Estimate for $I_3$  \\ 
			\hline
			\hline
			Case	1  & $\mathbb{P}_1^{(1)}$ & $\mathbb{P}_2^{(1)}$ & $\mathbb{P}_3^{(1)}$ \\
			\hline
			Case 2 & $\mathbb{P}_1^{(1)}$ & $\mathbb{P}_2^{(1)}$ & $\mathbb{P}_3^{(2)}$\\
			\hline
			Case 3 & $\mathbb{P}_1^{(1)}$ & $\mathbb{P}_2^{(2)}$ & $\mathbb{P}_3^{(1)}$\\
			\hline
			Case 4 & $\mathbb{P}_1^{(2)}$ & $\mathbb{P}_2^{(1)}$ & $\mathbb{P}_3^{(1)}$\\
			\hline
			Case 5 & $\mathbb{P}_1^{(2)}$ & $\mathbb{P}_2^{(2)}$ & $\mathbb{P}_3^{(2)}$\\
			\hline
			Case 6 & $\mathbb{P}_1^{(1)}$ & $\mathbb{P}_2^{(2)}$ & $\mathbb{P}_3^{(2)}$\\
			\hline
			Case 7 & $\mathbb{P}_1^{(2)}$ & $\mathbb{P}_2^{(1)}$ & $\mathbb{P}_3^{(2)}$\\
			\hline
			Case 8 & $\mathbb{P}_1^{(2)}$ & $\mathbb{P}_2^{(2)}$ & $\mathbb{P}_3^{(1)}$\\
			\hline
		\end{tabular}
		\caption{Combination of the estimates $I_1$, $I_2$ and $I_3$ for $n\geqslant 3$}
	\end{table}
	
	By straightforward computations, we find that Cases 1 and 5 are sufficient for us to prove blow-up result.
	
	More precisely, we may directly derive in Case 1
	\begin{align*}
	I+J+\int\nolimits_{\Omega}u_1(x)\phi_0(x)\varphi_T^\ell(x)\,\mathrm{d}x-\int\nolimits_{\Omega}u_0(x)\Delta\left(\phi_0(x)\varphi_T^\ell(x)\right)\mathrm{d}x\leqslant 3\varepsilon I+C\left(T^{n+1-2p'}+T^{n+1-3p'}\right).
	\end{align*}
	Let us take the next assumption:
	\begin{align}\label{Ass.Data}
	\int\nolimits_{\Omega}u_1(x)\phi_0(x)\,\mathrm{d}x>0.
	\end{align}
	Thus, we can conclude the blow-up of solution providing that
	\begin{align*}
	1<p\leqslant 1+\frac{2}{n-1}\quad\mbox{and}\quad 1<q.
	\end{align*}
	
	In Case 5, it is easy to get
	\begin{align*}
	&I+J+\int\nolimits_{\Omega}u_1(x)\phi_0(x)\varphi_T^\ell(x)\,\mathrm{d}x+\int\nolimits_{\Omega}u_0(x)\Psi_T(0)\Delta\left(\phi_0(x)\varphi^\ell_T(x)\right)\mathrm{d}x\\
	&\leqslant 3\varepsilon J+C\left(T^{n+1-2q'}+T^{n+1-q'}\right).
	\end{align*}
	Again, assuming \eqref{Ass.Data} for initial data, we get the blow-up result when
	\begin{align*}
	1<p\quad\mbox{and}\quad 1<q\leqslant 1+\frac{1}{n}.
	\end{align*}
	
	In conclusion, we complete the proof for $n\geqslant 3$.
	\medskip
	
	\noindent $\bullet$ \underline{Proof of blow-up for $n=2$.}
	\medskip
	
	\noindent Let us derive the estimates for $I_1$, $I_2$ and $I_3$ by employing Lemma \ref{lemma4}.\\
	\noindent \underline{Part I:} The possibilities of the estimates of $I_1$.
	\begin{align*}
	\mathbb{Q}_1^{(1)}:\quad &I_1=\int\nolimits_0^{T}\int\nolimits_{\Omega}u(t,x)\phi_0(x)\partial_t^2\left(\eta^k_T(t)\right)\varphi^\ell_T(x)\,\mathrm{d}x\,\mathrm{d}t\leqslant \varepsilon I+C\ln(T)T^{n+1-2p'},\\
	\mathbb{Q}_1^{(2)}:\quad &I_1=-\int\nolimits_0^{T}\int\nolimits_{\Omega}u_t(t,x)\phi_0(x)\partial_t\left(\eta^k_T(t)\right)\varphi^\ell_T(x)\,\mathrm{d}x\,\mathrm{d}t\leqslant\varepsilon J+C\ln(T)T^{n+1-q'}.
	\end{align*}
	\noindent \underline{Part II:} The possibilities of the estimates of $I_2$.
	\begin{align*}
	\mathbb{Q}_2^{(1)}:\quad &I_2=\int\nolimits_0^{T}\int\nolimits_{\Omega}u(t,x)\partial_t\left(\eta^k_T(t)\right)\Delta\left(\phi_0(x)\varphi^\ell_T(x)\right)\mathrm{d}x\,\mathrm{d}t  \leqslant \varepsilon I+C\ln(T)T^{n+1-3p'}+C\,T^{n+1-3p'},\\
	\mathbb{Q}_2^{(2)}:\quad  &I_2=-\int\nolimits_0^{T}\int\nolimits_{\Omega}u_t(t,x)\eta^k_T(t)\Delta\left(\phi_0(x)\varphi^\ell_T(x)\right)\mathrm{d}x\,\mathrm{d}t-\int\nolimits_{\Omega}u_0(x)\Delta\left(\phi_0(x)\varphi^\ell_T(x)\right)\mathrm{d}x\\
	&\quad\leqslant \varepsilon J+C\ln(T)T^{n+1-2q'}+C\,T^{n+1-2q'}-\int\nolimits_{\Omega}u_0(x)\Delta\left(\phi_0(x)\varphi^\ell_T(x)\right)\mathrm{d}x.
	\end{align*}
	\noindent \underline{Part III:} The possibilities of the estimates of $I_3$.
	\begin{align*}
	\mathbb{Q}_3^{(1)}:\quad &I_3=-\int\nolimits_0^{T}\int\nolimits_{\Omega}u(t,x)\eta^k_T(t)\Delta\left(\phi_0(x)\varphi^\ell_T(x)\right)\mathrm{d}x\,\mathrm{d}t  \leqslant \varepsilon I+C\,T^{n+1-2p'}+C\ln(T)T^{n+1-2p'},\\
	\mathbb{Q}_3^{(2)}:\quad  &I_3=-\int\nolimits_0^{T}\int\nolimits_{\Omega}u_t(t,x)\Psi_T(t)\Delta\left(\phi_0(x)\varphi^\ell_T(x)\right)\mathrm{d}x\,\mathrm{d}t-\int\nolimits_{\Omega}u_0(x)\Psi_T(0)\Delta\left(\phi_0(x)\varphi^\ell_T(x)\right)\mathrm{d}x\\
	&\quad\leqslant \varepsilon J+C\,T^{n+1-q'}+C\ln(T)T^{n+1-q'}-\int\nolimits_{\Omega}u_0(x)\Psi_T(0)\Delta\left(\phi_0(x)\varphi^\ell_T(x)\right)\mathrm{d}x.
	\end{align*}
	In the above,  the positive constant is chosen by $\varepsilon\in(0,1/6)$.\\
	Now, we need to distinguish between the following eight cases in the next table.\\
	\begin{table}[http]
		\centering
		\begin{tabular}{|c | c | c| c |}
			\hline
			& Estimate for $I_1$ & Estimate for $I_2$  & Estimate for $I_3$  \\ 
			\hline
			\hline
			Case	1  & $\mathbb{Q}_1^{(1)}$ & $\mathbb{Q}_2^{(1)}$ & $\mathbb{Q}_3^{(1)}$ \\
			\hline
			Case 2 & $\mathbb{Q}_1^{(1)}$ & $\mathbb{Q}_2^{(1)}$ & $\mathbb{Q}_3^{(2)}$\\
			\hline
			Case 3 & $\mathbb{Q}_1^{(1)}$ & $\mathbb{Q}_2^{(2)}$ & $\mathbb{Q}_3^{(1)}$\\
			\hline
			Case 4 & $\mathbb{Q}_1^{(2)}$ & $\mathbb{Q}_2^{(1)}$ & $\mathbb{Q}_3^{(1)}$\\
			\hline
			Case 5 & $\mathbb{Q}_1^{(2)}$ & $\mathbb{Q}_2^{(2)}$ & $\mathbb{Q}_3^{(2)}$\\
			\hline
			Case 6 & $\mathbb{Q}_1^{(1)}$ & $\mathbb{Q}_2^{(2)}$ & $\mathbb{Q}_3^{(2)}$\\
			\hline
			Case 7 & $\mathbb{Q}_1^{(2)}$ & $\mathbb{Q}_2^{(1)}$ & $\mathbb{Q}_3^{(2)}$\\
			\hline
			Case 8 & $\mathbb{Q}_1^{(2)}$ & $\mathbb{Q}_2^{(2)}$ & $\mathbb{Q}_3^{(1)}$\\
			\hline
		\end{tabular}
		\caption{Combination of the estimates $I_1$, $I_2$ and $I_3$ for $n=2$}
	\end{table}
	
	Actually, in order to prove our result, we just need to consider Cases 1 and 5, respectively.
	
	For one thing, concerning Case 1, we may obtain the estimate
	\begin{align*}
	&I+J+\int\nolimits_{\Omega}u_1(x)\phi_0(x)\varphi_T^\ell(x)\,\mathrm{d}x-\int\nolimits_{\Omega}u_0(x)\Delta\left(\phi_0(x)\varphi_T^\ell(x)\right)\mathrm{d}x\\
	&\leqslant 3\varepsilon I+C\left(T^{n+1-2p'}+T^{n+1-3p'}\right)+C\ln(T)\left(T^{n+1-2p'}+T^{n+1-3p'}\right).
	\end{align*}
	By assumption \eqref{Ass.Data}, the blow-up of solution can be derived if
	\begin{align*}
	1<p<1+\frac{2}{n-1}\quad \mbox{and}\quad 1<q.
	\end{align*}
	
	For another, taking the consideration of Case 5, one has
	\begin{align*}
	&I+J+\int\nolimits_{\Omega}u_1(x)\phi_0(x)\varphi_T^\ell(x)\,\mathrm{d}x+\int\nolimits_{\Omega}u_0(x)\Psi_T(0)\Delta\left(\phi_0(x)\varphi^\ell_T(x)\right)\mathrm{d}x\\
	&\leqslant 3\varepsilon J+C\left(T^{n+1-2q'}+T^{n+1-q'}\right)+C\ln(T)\left(T^{n+1-2q'}+T^{n+1-q'}\right).
	\end{align*}
	With the assumption \eqref{Ass.Data}, the solution blows up providing that
	\begin{align*}
	1<p\quad\mbox{and}\quad1<q<1+\frac{1}{n}.
	\end{align*}
	
	Then, they complete the proof in the case $n=2$.

	\medskip
	
	\noindent $\bullet$ \underline{Proof of blow-up for $n=1$.}
	\medskip
	
	\noindent The application of Lemma \ref{lemma2} with the change of variables such that $y=T^{-\alpha}x$ and $s=T^{-1}t$, where $\alpha>1$, implies the next estimates.\\
	\noindent \underline{Part I:} The possibilities of the estimates of $I_1$.
	\begin{align*}
	\mathbb{W}_1^{(1)}:\quad &I_1=\int\nolimits_0^{T}\int\nolimits_{\Omega}u(t,x)\phi_0(x)\partial_t^2\left(\eta^k_T(t)\right)\varphi^\ell_T(x)\,\mathrm{d}x\,\mathrm{d}t\leqslant \varepsilon I+C\,T^{-2p'+2\alpha+1},\\
	\mathbb{W}_1^{(2)}:\quad &I_1=-\int\nolimits_0^{T}\int\nolimits_{\Omega}u_t(t,x)\phi_0(x)\partial_t\left(\eta^k_T(t)\right)\varphi^\ell_T(x)\,\mathrm{d}x\,\mathrm{d}t\leqslant\varepsilon J+C\,T^{-q'+2\alpha+1}.
	\end{align*}
	
	\noindent \underline{Part II:} The possibilities of the estimates of $I_2$.
	\begin{align*}
	\mathbb{W}_2^{(1)}:\,\,\,\, &I_2=\int\nolimits_0^{T}\int\nolimits_{\Omega}u(t,x)\partial_t\left(\eta^k_T(t)\right)\partial_x^2\left(\phi_0(x)\varphi^\ell_T(x)\right)\mathrm{d}x\,\mathrm{d}t  \leqslant\varepsilon I+C\,T^{-(\alpha+1) (p'-1)}+C\,T^{-(2\alpha+1) (p'-1)},\\
	\mathbb{W}_2^{(2)}:\,\,\,\,  &I_2=-\int\nolimits_0^{T}\int\nolimits_{\Omega}u_t(t,x)\eta^k_T(t)\partial_x^2\left(\phi_0(x)\varphi^\ell_T(x)\right)\mathrm{d}x\,\mathrm{d}t-\int\nolimits_{\Omega}u_0(x)\partial_x^2\left(\phi_0(x)\varphi^\ell_T(x)\right)\mathrm{d}x\\
	&\quad\leqslant \varepsilon J+C\,T^{-\alpha q'+\alpha+1}+C\,T^{-2\alpha q'+2\alpha+1}-\int\nolimits_{\Omega}u_0(x)\partial_x^2\left(\phi_0(x)\varphi^\ell_T(x)\right)\mathrm{d}x.
	\end{align*}
	
	\noindent \underline{Part III:} The possibilities of the estimates of $I_3$.
	\begin{align*}
	\mathbb{W}_3^{(1)}:\quad &I_3=-\int\nolimits_0^{T}\int\nolimits_{\Omega}u(t,x)\eta^k_T(t)\partial_x^2\left(\phi_0(x)\varphi^\ell_T(x)\right)\mathrm{d}x\,\mathrm{d}t  \leqslant\varepsilon I+C\,T^{-\alpha p'+\alpha+1}+C\,T^{-2\alpha p'+2\alpha+1},\\
	\mathbb{W}_3^{(2)}:\quad  &I_3=-\int\nolimits_0^{T}\int\nolimits_{\Omega}u_t(t,x)\Psi_T(t)\partial_x^2\left(\phi_0(x)\varphi^\ell_T(x)\right)\mathrm{d}x\,\mathrm{d}t-\int\nolimits_{\Omega}u_0(x)\Psi_T(0)\partial_x^2\left(\phi_0(x)\varphi^\ell_T(x)\right)\mathrm{d}x\\
	&\quad\leqslant\varepsilon J+C\,T^{-(\alpha-1)q'+\alpha+1}+C\,T^{-(2\alpha-1)q'+2\alpha+1}-\int\nolimits_{\Omega}u_0(x)\Psi_T(0)\partial_x^2\left(\phi_0(x)\varphi^\ell_T(x)\right)\mathrm{d}x.
	\end{align*}
	We may take the positive constant by $\varepsilon\in(0,1/6)$ in the above estimates.\\
	
	In the next table, let us now distinguish between the following eight cases.
	\begin{table}[http]
		\centering
		\begin{tabular}{|c | c | c| c |}
			\hline
			& Estimate for $I_1$ & Estimate for $I_2$  & Estimate for $I_3$  \\ 
			\hline
			\hline
			Case	1  & $\mathbb{W}_1^{(1)}$ & $\mathbb{W}_2^{(1)}$ & $\mathbb{W}_3^{(1)}$ \\
			\hline
			Case 2 & $\mathbb{W}_1^{(1)}$ & $\mathbb{W}_2^{(1)}$ & $\mathbb{W}_3^{(2)}$\\
			\hline
			Case 3 & $\mathbb{W}_1^{(1)}$ & $\mathbb{W}_2^{(2)}$ & $\mathbb{W}_3^{(1)}$\\
			\hline
			Case 4 & $\mathbb{W}_1^{(2)}$ & $\mathbb{W}_2^{(1)}$ & $\mathbb{W}_3^{(1)}$\\
			\hline
			Case 5 & $\mathbb{W}_1^{(2)}$ & $\mathbb{W}_2^{(2)}$ & $\mathbb{W}_3^{(2)}$\\
			\hline
			Case 6 & $\mathbb{W}_1^{(1)}$ & $\mathbb{W}_2^{(2)}$ & $\mathbb{W}_3^{(2)}$\\
			\hline
			Case 7 & $\mathbb{W}_1^{(2)}$ & $\mathbb{W}_2^{(1)}$ & $\mathbb{W}_3^{(2)}$\\
			\hline
			Case 8 & $\mathbb{W}_1^{(2)}$ & $\mathbb{W}_2^{(2)}$ & $\mathbb{W}_3^{(1)}$\\
			\hline
		\end{tabular}
		\caption{Combination of the estimates $I_1$, $I_2$ and $I_3$ for $n=1$}
	\end{table}
	
	Indeed, it is enough for us to focus on Cases 1, 2, 4 and 5 by straightforward calculations.
	
	Precisely, we may have the estimate in Case 1 that
	\begin{align*}
	&I+J+\int\nolimits_{\Omega}u_1(x)\phi_0(x)\varphi_T^\ell(x)\,\mathrm{d}x-\int\nolimits_{\Omega}u_0(x)\partial_x^2\left(\phi_0(x)\varphi_T^\ell(x)\right)\mathrm{d}x\\
	&\leqslant 3\varepsilon I+C\left(T^{-2p'+2\alpha+1}+T^{-(\alpha+1) (p'-1)}+T^{-\alpha p'+\alpha+1}\right).
	\end{align*}
	Consequently, with the assumption \eqref{Ass.Data} the solution blows up in finite time when
	\begin{align*}
	1<p\leqslant 1+\alpha_1\approx 2.28 \quad\mbox{and}\quad 1<q.
	\end{align*}
	
	Next, the simple calculation in Case 2 implies
	\begin{align*}
	&I+J+\int\nolimits_{\Omega}u_1(x)\phi_0(x)\varphi_T^\ell(x)\,\mathrm{d}x-\int\nolimits_{\Omega}u_0(x)\partial_x^2\left(\phi_0(x)\varphi_T^\ell(x)\right)\mathrm{d}x+\int\nolimits_{\Omega}u_0(x)\Psi_T(0)\partial_x^2\left(\phi_0(x)\varphi^\ell_T(x)\right)\mathrm{d}x\\
	&\leqslant 2\varepsilon I+\varepsilon J+C\left(T^{-2p'+2\alpha+1}+T^{-(\alpha+1) (p'-1)}+ T^{-(\alpha-1)q'+\alpha+1}+T^{-(2\alpha-1)q'+2\alpha+1}\right).
	\end{align*}
	One may claim blow-up of solutions under the assumption \eqref{Ass.Data} and the condition
	\begin{align*}
	1<p\leqslant \frac{2\alpha+1}{2\alpha-1}\quad\mbox{and}\quad 1<q\leqslant \frac{\alpha+1}{2}.
	\end{align*}
	
	In Case 4, we compute
	%
	\begin{align*}
	&I+J+\int\nolimits_{\Omega}u_1(x)\phi_0(x)\varphi_T^\ell(x)\,\mathrm{d}x-\int\nolimits_{\Omega}u_0(x)\partial_x^2\left(\phi_0(x)\varphi_T^\ell(x)\right)\mathrm{d}x\\
	&\leqslant 2\varepsilon I+\varepsilon J+C\left(T^{-q'+2\alpha+1}+T^{-(\alpha+1) (p'-1)}+T^{-\alpha p'+\alpha+1}\right).
	\end{align*}
	Moreover, by considering the assumption \eqref{Ass.Data}, the solution blows up in finite time when
	\begin{align*}
	1<p\leqslant \alpha+1\quad\mbox{and}\quad 1<q\leqslant \frac{2\alpha+1}{2\alpha}.
	\end{align*}
	
	Finally, we may obtain the estimate in case 5
	\begin{align*}
	&I+J+\int\nolimits_{\Omega}u_1(x)\phi_0(x)\varphi_T^\ell(x)\,\mathrm{d}x+\int\nolimits_{\Omega}u_0(x)\Psi_T(0)\partial_x^2\left(\phi_0(x)\varphi^\ell_T(x)\right)\mathrm{d}x\\
	&\leqslant 3\varepsilon J+C\left(T^{-q'+2\alpha+1}+T^{-\alpha q'+\alpha+1}+T^{-(\alpha-1)q'+\alpha+1}\right).
	\end{align*}
	We conclude blow-up of solution providing that the assumption \eqref{Ass.Data} is fulfilled and
	\begin{align*}
	1<p \quad\mbox{and}\quad 1<q\leqslant \frac{1+\alpha_2}{2}\approx 1.3.
	\end{align*}

	Again, it is enough to consider just Case 1, Case 2 (for $\alpha<\alpha_1$), Case 4 (for $\alpha_1<\alpha<\alpha_2$), and Case 5. Combining with the results from these four cases, we immediately complete the proof in the case $n=1$.
\end{proof}

\subsection*{Acknowledgments} The Ph.D. study of Wenhui Chen is supported by S\"achsiches Landesgraduiertenstipendium.



\begin{thebibliography}{99}
\bibitem{DAbbiccoReissig2014} D'Abbicco M, Reissig M. Semilinear structural damped waves. \emph{Math Methods Appl Sci.} 2014;37(11):1570-1592.
%
\bibitem{D'AmbrosioLucente2003} D'Ambrosio L, Lucente S. Nonlinear Liouville theorems for Grushin and Tricomi operators. \emph{J Differential Equations} 2003;193(2):511-541.
%
\bibitem{Fino} Fino AZ. Finite time blow up for wave equations with strong damping in an exterior domain. \emph{Preprint} arXiv: 2695271. 
%
\bibitem{Finogeorgiev}  Fino AZ,  Georgiev V,  Kirane M. Finite time blow-up for a wave equation with a nonlocal nonlinearity. \emph{Preprint} arXiv:1008.4219.
%
\bibitem{FinoKarch} Fino A,  Karch G. Decay of mass for nonlinear equation with fractional Laplacian.  \emph{Monatsh Math.} 2010;160(4):375-384.
%
\bibitem{Finokirane} Fino AZ, Kirane M. Qualitative properties of solutions to a time-space fractional evolution equation. \emph{Quart Appl Math.} 2012;70(1):133-157.
%
\bibitem{Han} Han W. Concerning the Strauss conjecture for the subcritical and critical cases on the exterior domain in two space dimensions. \emph{Nonlinear Anal.} 2013;84:136-145.
%
\bibitem{Han2} Han W. Blow up of solutions to one dimensional initial-boundary value problems for semilinear wave equations with variable coefficients. \emph{J Partial Differ Equ.} 2013;26(2):138-150.
%
\bibitem{Ikehata} Ikehata R. Decay estimates of solutions for the wave equations with strong damping terms in unbounded domains. \emph{Math Methods Appl Sci.} 2001;24(9):659-670. 
%
\bibitem{Ikehata2014} Ikehata R. Asymptotic profiles for wave equations with strong damping. \emph{J Differential Equations} 2014;257(6):2159-2177.
%
\bibitem{IkehataTodorovaYordanov2013} Ikehata R, Todorova G, Yordanov B. Wave equations with strong damping in Hilbert spaces. \emph{J Differential Equations} 2013;254(8):3352-3368.
%
\bibitem{IkehataInoue} Ikehata R, Inoue Y. Global existence of weak solutions for two-dimensional semilinear wave equations with strong damping in an exterior domain. \emph{Nonlinear Anal.} 2008;68(1):154-169.
%
\bibitem{KPS} Kobayashi T,  Pecher H, Shibata Y. On a global in time existence theorem of smooth solutions to a nonlinear wave equation with viscosity. \emph{Math Ann.} 1993;296(2):215-234. 
%
\bibitem{PM1}  Mitidieri \'E, Pohozaev SI.  A priori estimates and the absence of solutions of nonlinear partial differential equations and inequalities. \emph{Proc Steklov Inst Math.} 2001;234(3):1-362.
%
\bibitem{Ponce1985} Ponce G. Global existence of small solutions to a class of nonlinear evolution equations. \emph{Nonlinear Anal.} 1985;9(5):399-418.
%
\bibitem{Shibata2000} Shibata Y.  On the rate of decay of solutions to linear viscoelastic equation. \emph{Math Methods Appl Sci.} 2000;23(3):203-226.
%
\bibitem{Yuta} Wakasugi Y. On the diffusive structure for the damped wave equation with variable coefficients. Doctoral Thesis, Osaka University; 2014.
%
\bibitem{Zhang} Zhang QS. A blow-up result for a nonlinear wave equation with damping: the critical case. \emph{C R Acad Sci Paris S\'er I Math.} 2001;333(2):109-114.
%
\bibitem{ZhouHan} Zhou Y, Han W. Blow-up of solutions to semilinear wave equations with variable coefficients and boundary. \emph{J Math Anal Appl.} 2011;374(2):585-601.
\end{thebibliography}
\end{document}